\documentclass[12pt]{article}

\usepackage{amsmath, amssymb, amsthm, hyperref}
\usepackage{xcolor}
\usepackage{cleveref}

\usepackage{enumerate}
\usepackage{enumitem}

\def\forkindep{\mathrel{\raise0.2ex\hbox{\ooalign{\hidewidth$\vert$\hidewidth\cr\raise-0.9ex\hbox{$\smile$}}}}}

\theoremstyle{definition}
\newtheorem{thm}{Theorem}[section]
\newtheorem{cor}[thm]{Corollary}
\newtheorem{prop}[thm]{Proposition}
\newtheorem{lem}[thm]{Lemma}
\newtheorem{fact}[thm]{Fact}

\theoremstyle{definition}
\newtheorem{dfn}[thm]{Definition}
\newtheorem{exm}[thm]{Example}
\newtheorem{rem}[thm]{Remark}
\newtheorem{theo}[thm]{Theorem}
\newtheorem{obser}[thm]{Observation}

\newtheorem*{thm*}{Theorem}

\DeclareMathOperator{\acl}{acl}
\DeclareMathOperator{\dcl}{dcl}
\DeclareMathOperator{\Acl}{Acl}
\DeclareMathOperator{\tp}{tp}

\DeclareMathOperator{\cl}{cl}
\DeclareMathOperator{\qftp}{qftp}

\DeclareMathOperator{\CL}{CL}

\DeclareMathOperator{\indep}{indep}

\usepackage[symbol]{footmisc}

\begin{document}
\title{Bi-Colored Expansions of Geometric Theories }
\author{S. Jalili\footnote{Department of Mathematics and Computer Science, Amirkabir University of Technology (Tehran Polytechnic), Hafez Avenue 15194, P.O.Box 15875-4413,
	Tehran, Iran. E-mail: somaye.jalili507@gmail.com.},
M.  Pourmahdian\footnote{Department of Mathematics and Computer Science, Amirkabir University of Technology (Tehran Polytechnic), Hafez Avenue 15194, P.O.Box 15875-4413, Tehran, Iran; School of Mathematics, Institute for Research in Fundamental Sciences (IPM), P.O. Box 19395-5746, Tehran, Iran. E-mail: pourmahd@ipm.ir.},
M.  Khani\footnote{ Department of Mathematical Sciences, Isfahan University of
Technology, Isfahan, 84156-83111, Iran, mohsen.khani@iut.ac.ir}. \\
}
\date{}
\maketitle
\begin{abstract}
This paper concerns the study of Bi-colored expansions of geometric theories  in the light of the Fra\"{i}ss\'{e}-Hrushovski construction method.
Substructures of models of a geometric theory $T$ are expanded by a color predicate $p$, and
the dimension function associated with the pre-geometry of the $T$-algebraic closure operator together with a real number $0<\alpha\leqslant 1$ is used to define a pre-dimension function $\delta_{\alpha}$. The pair  $(\mathcal{K}_{\alpha}^{+},\leqslant_{\alpha})$ consisting of all such expansions with a hereditary positive pre-dimension along with the notion of substructure  $\leqslant_{\alpha}$ associated to $\delta_{\alpha}$
is then used as a natural setting for the study of generic bi-colored expansions in the style of Fra\"{i}ss\'{e}-Hrushovski construction.  Imposing certain natural conditions on $T$, enables us to introduce a complete axiomatization $\mathbb{T}_{\alpha}$ for the class of rich structures in this class. We will show that if $T$ is a dependent theory  (NIP) then so is $\mathbb{T}_{\alpha}$. We further prove that whenever $\alpha$ is rational the strong dependence transfers to $\mathbb{T}_{\alpha}$. We conclude by showing that if $T$ defines a linear order and $\alpha$ is irrational then $\mathbb{T}_{\alpha}$ is not strongly dependent.
\\
\textit{Keywords:} Geometric Theories; Fra\"{i}ss\'{e}-Hrushovski construction; Bi-Colored Expansions; Rich Structures; Dependence (NIP); Strong Dependence.
\\
\textit{AMS Subject Classification 2020} 03C45.
\end{abstract}
%\tableofcontents
\section{Introduction}
Extending  methods  of stability and applying them to a larger class of theories is a dominating theme in current research of pure model theory.
In addition to showing the prevalence  of these methods, this line of research also presents potential applications beyond model theory. To instantiate  general concepts in stability-hierarchy and perhaps examine some related open questions/conjectures one would need to look for some new examples, conceivably through  adapting known model-theoretic methods.
%From this prospective,
Our main aim in this paper
 is to study the Fra\"{i}ss\'{e}-Hrushovski method beyond the realm of stability/simplicity.   We aim to take the first step of the more general plan of modifying the Fra\"{i}ss\'{e}-Hrushovski construction to study Bi-colored expansions of geometric theories without the independence property (NIP).
\par
 There are two  major (independent) research themes with which our proposed plan is naturally connected, as explained below.
 \par
  Our first source of contribution is to the Bi-colored fields, that is the comprehensive studies concerning the expansion of algebraically closed fields with a unary predicate $p$---often called a color predicate--- interpreted either  by an arbitrary set (Black fields)
 \cite{BH00,BH01,B99}, an additive subgroup (Red fields)
 \cite{BPZ07,B01}, or a multiplicative subgroup (Green fields)
 \cite{BMPWZ09,B01}. All examples obtained by this construction are $\omega$-stable, either with Morley rank $\omega$ (non-collapsed constructions)\cite{BH01,B99,B01} or finite Morley rank (collapsed constructions)\cite{BH00,BMPWZ09}.
 \par
The other theme that our results are naturally connected to is the study of the generic expansions of models of geometric theories by a unary predicate which can be interpreted either by a submodel (Lovely pairs)\cite{BDO11,BV10,B83,Y02} or more generally by  submodels of reducts \cite{BV20,CP98,DG06,E21}.
\par
While our methodology is inspired by the first subject,  in our study of the  bi-colored expansions  of theories, we do not restrict ourselves, a priori, to any particular class of theories in stability hierarchy. This would give a clearer understanding of the axiomatization of the resulting expansion in the style of the second subject. Our presentation somehow shows that both subjects could be put in a common framework where the  Fra\"{i}ss\'{e}-Hrushovski construction is seen as a  sort of generalization of the theme of the generic expansions.

In more technical terms,
we assume that $T$ is a complete geometric theory without finite models in a countable language
$\mathcal{L}$.
As a convenient routine, we also assume that
$T$
admits elimination of quantifiers in
$\mathcal{L}$.
As $T$ is geometric, the algebraic closure operator gives rise to a pre-geometry and subsequently a notion of dimension function $\dim$, where  $\dim(\bar{a})$ for any finite tuple $\bar{a}$  is the size of a basis of $\bar{a}$. We further require that $T$ satisfies the free-amalgamation property and its corresponding pre-geometry is \textit{indecomposable}.  Subsequently, we expand
$\mathcal{L}$
to
$\mathcal{L}_{p}=\mathcal{L}\cup \{p\}$
by adding a unary predicate
$p$
called the \textit{coloring predicate}.
We consider
the class of all
$\mathcal{L}_{p}$-structures
whose universe
$M$ is a model of $T^{\forall}$, fix a real number
$0<\alpha \leqslant 1$, and for
$M\models T^{\forall}$
and a finite subset
$A$ of
$M$, define the pre-dimension function
\begin{align*}
\delta_{\alpha}(A)=\dim(A)-\alpha|p(A)|.
\end{align*}
Now, as $T$ is geometric, the dimension function enjoys certain definability conditions which makes  $\mathcal{K}_{\alpha}^{+}$,
the class of $\mathcal{L}_{p}$-structures
$M\in\mathcal{K}_{\alpha}^{+}$
such that
$\delta_{\alpha}(A)\geqslant 0$
for all
$A\subseteq_{\text{fin}}M$,
$\mathcal{L}_{p}$-axiomatizable.
\par
There is a notion $\leqslant_{\alpha}$ of closed substructures in $\mathcal{K}_{\alpha}^{+}$ associated to the pre-dimension function $\delta_{\alpha}$. The free-amalgamation of $T$ implies that $(\mathcal{K}_{\alpha}^{+},\leqslant_{\alpha})$ has the amalgamation property, and this guaranties that $(\mathcal{K}_{\alpha}^{+},\leqslant_{\alpha})$ has $\lambda$-rich models, for each $\lambda \geqslant \aleph_0$.
\par
Using ideas from \cite{L07}, we give a complete axiomatization $\mathbb{T}_{\alpha}$ for the class all $\omega_1$-rich structures. This axiomatization together with a description of types enables us to characterize  coheirs in models of  $\mathbb{T}_{\alpha}$ and  prove that if $T$ is dependent then so is  $\mathbb{T}_{\alpha}$.
\par
Yet we will show that our construction separates between rational and irrational $\alpha$ via the notion of strong dependence.
The strong dependence is transferred  to $\mathbb{T}_{\alpha}$ for rational $\alpha$,  but this is not the case when $\alpha$ is irrational. Indeed, assuming that $T$ defines a linear order and $\alpha$ is irrational we show that $\mathbb{T}_{\alpha}$ is not strongly dependent. This fact mimics the well-known result in ab-initio Fra\"{i}ss\'{e}-Hrushovski construction that differentiate between rational/irrational $\alpha$ via $\omega$-stability/strict-stability.
\par
The structure of the paper is as follows. In section 2 after fixing the setting for our underlying theory $T$, we introduce the class $(\mathcal{K}^+_{\alpha},\leqslant_{\alpha})$.
In Section 3 we prove certain properties of this class, present an axiomatization
$\mathbb{T}_\alpha$ for its rich structures, and prove it to be a complete theory.
Finally in Section 4 we have given the mentioned results on dependence and strong dependence of $\mathbb{T}_\alpha$.
\section{Bi-Colored Structures}
\subsection{Preliminaries, Setting and Conventions}
Let
$T$
be
as in the introduction.
It is clear, by the quantifier-elimination
assumption,
that
$\acl_{M_{1}}(M_{0})\cong \acl_{M_{2}}(M_{0})$ for all $M_0\subseteq M_1,M_2\models T$.
\par \noindent
We use capital letters
$M, N, P,K$ to denote the
$\mathcal{L}$-structures and $ A, B, C,D$
and $X, Y, Z$
respectively for
 finite and infinite sets.
Instead of
$X\cup Y$ we would write $XY$. For tuples $\bar{a},\bar{b}$ in a model of $M$ of $T$,  by the {\em quantifier-free type} of $\bar{a}$ over $\bar{b}$, denoted by $\qftp_{\mathcal{L}}(\bar{a}/\bar{b})$, we mean the set of all quantifier-free formulas
$\varphi(\bar{x},\bar{b})$  such that $M\models \varphi(\bar{a},\bar{b})$.
\par
The theory $T$
is assumed to be \textit{geometric}, so
it
eliminates the quantifier
$\exists^{\infty}$
and the algebraic closure, $\acl$, satisfies the exchange property in its models.
We denote the
dimension obtained by
$\acl$ in
$T$ by $\dim$. So
 for  tuples $\bar{a}=(a_{1},...,a_{n})$ and $\bar{b}=(b_1,\dots,b_m)$ of $M$ and
 a set $Y\subseteq M$, we let $\dim(\bar{a}/Y)$
be the size of the $\acl$-base of $\{a_1,\ldots,a_n\}$ over $Y$,
and $\dim(\bar{a}/\bar{b})=\dim(\bar{a}/\{b_1,\dots,b_m\})$.
 If $\varphi(\bar{x},\bar{y})$ is an $\mathcal{L}$-formula and $\bar{b}\in M^{|\bar{y}|}$, then
  $\dim(\varphi(M,\bar{b}))=\max \{ \dim(\bar{a}/\bar{b}):\hspace*{3mm}\ M\models \varphi(\bar{a},\bar{b})\}$.
\par
We say that
$Y$
is
$\dim$-independent from
$Z$
over
$X$, and write
$Y\forkindep_{X}^{\dim} Z$,
if
$ \dim(\bar{a}/X) =\dim(\bar{a}/XZ)$
for each tuple
$\bar{a}\in Y^{|\bar{a}|}$.
If in addition, $Y\cap Z=X$, we
say that $Y$ and $Z$ are {\em free} over $X$.  To emphasize the freeness of
$Y$ and $Z$ over $X$ we would write the set $YZ$ as $Y\oplus_{X} Z$ (we also use the same notation for the free-amalgam of structures).
%and  the $\mathcal{L}$-structure generated by $YZ$ is called a {\em free amalgam} of $Y$ and $Z$ over $X$.
\par
Some  properties of
the $\dim$  function  and $\dim$-independence
is recalled in the following for later reference, assuming that
all subsets and tuples are from (a sufficiently  saturated) model
$M$ of
$T$.
\begin{fact}\label{fact:1.1}
The $\dim$ function has the following properties.
\begin{itemize}
\item\textit{Finite character.}
$\dim(\bar{a}/Y)=\min\{\dim(\bar{a}/B):\hspace*{3mm}B\subseteq_{\text{fin}} Y\}$.
\item \textit{Additivity.}
$\dim(\bar{a}\bar{b}/Z) = \dim(\bar{a}/ Z) + \dim(\bar{b}/\bar{a} Z)$.
\item\textit{Monotonicity.}
$\dim(\bar{a}/Y) \geqslant \dim(\bar{a}/Z)$
for
$Y \subseteq Z$.
\item \textit{Definability.}
For each
formula
$\varphi(\bar{x},\bar{y})$
and
$k\leqslant |\bar{x}|$ the set of tuples
$\bar{b}\in M^{|\bar{y}|}$ with  $\dim (\varphi(M, \bar{b}))=k$ is definable by a formula
$d_{\varphi,k}(\bar{y})$.

\item \textit{$\bigvee$-Definability.}
If
$\dim(\bar{a}/\bar{b})\leqslant n$,
then there is
a formula $\psi(\bar{x},\bar{y})$ such that
$\psi(\bar{x},\bar{b})\in \qftp_{\mathcal{L}}(\bar{a}/\bar{b})$
and  if
$M\models \psi(\bar{a}^{'},\bar{b}^{'})$ then
$\dim(\bar{a}^{'}/\bar{b}^{'})\leqslant n$, for any
$\bar{a}^{'},\bar{b}^{'}$.
\end{itemize}
\end{fact}
We need a stronger version of Definability, as stated below.
The first part of the following statement  appears in \cite{CP98}, Lemma 2.3.
\begin{lem}\label{defacl}
Let $M$ be an $\omega$-saturated model of $T$, $\varphi(\bar{x};\bar{y})$ an
$\mathcal{L}$-formula and
 $k$ a natural number. Then
\begin{enumerate}
\item
There is a formula
$H_\varphi(\bar{y})$ that defines
the set of tuples
 $\bar{b}\in M^{|\bar{y}|}$
 such that there exists
 $\bar{a}$
 satisfying
  $\varphi(\bar{x};\bar{b})$
  with
   $\bar{a}\cap \acl(\bar{b})=\emptyset$.
\item
There exists an
 $\mathcal{L}$-formula
 $D_{\varphi,k}(\bar{y})$
 which defines the set of tuples
 $\bar{b}\in M^{|\bar{y}|}$
 for which there exists
$\bar{a}$
 satisfying
  $\varphi(\bar{x};\bar{b})$
 with
  $\dim(\bar{a}/\bar{b})\geqslant  k$
  and
 $\bar{a}\cap \acl(\bar{b})=\emptyset$.
\end{enumerate}
\end{lem}

\begin{proof}
We use 1 to prove 2.
Certainly, we may assume that
 $k\leqslant |\bar{x}|=n$.
 For
 $\sigma\in S_n$
 and
  $\psi(x_{\sigma(1)}; x_{\sigma(2)},\dots, x_{\sigma(n)},\bar{y})=\varphi(\bar{x},\bar{y})$
 the assertion
$$\exists x_{\sigma(1)} \big[\varphi(\bar{x},\bar{y}) \wedge  x_{\sigma(1)}\notin \acl(x_{\sigma(2)},\dots, x_{\sigma(n)},\bar{y})\big]$$
is definable by a formula $H_{\psi}(x_{\sigma(2)};\dots,x_{\sigma(n)},\bar{y}).$ Now by applying $1$ to $H_{\psi}$ and iterating this procedure the following statement
\begin{align*}
\exists x_{\sigma(k)}\dots \exists x_{\sigma(1)}&\big[ \varphi(\bar{x},\bar{y})\wedge \bigwedge_{i=1}^{k}x_{\sigma(i)} \notin \acl(x_{\sigma(i+1)},\dots,x_{\sigma(k+1)},\dots,x_{\sigma(n)},\bar{y})\big]
\end{align*}
is definable by an
$\mathcal{L}$-formula
 $\rho(x_{\sigma(k+1)},\dots,x_{\sigma(n)};\bar{y}).$
 Informally $\rho$ states that
there  exists
a set
 $\{x_{\sigma(1)}, \dots,x_{\sigma(k)} \}\subseteq \{x_1,\dots,x_n\}$ which has dimension at least
 $k$
 over
 $\bar{y}$.
  Now by applying $1$ to $\rho$  we get that the statement
 $$\exists x_{\sigma(k+1)}\dots \exists x_{\sigma(n)}\big(\rho(x_{\sigma(k+1)},\dots,x_{\sigma(n)};\bar{y})\wedge \bigwedge_{i=k+1}^n x_{\sigma(i)}\notin \acl(\bar{y})\big)$$
is definable by an
$\mathcal{L}$-formula
$D_{\varphi,k,\sigma}(\bar{y})$.
Let $D_{\varphi,k}(\bar{y})=\bigvee_{\sigma\in S_n} D_{\varphi,k,\sigma}(\bar{y})$.
Note that any tuple $\bar{b}$ in $M$ satisfies
  $D_{\varphi,k}(\bar{y})$ if and only if there exists a tuple
 $\bar{a}$
which satisfies
 $\varphi(\bar{x},\bar{b})$,
 avoids
  $\acl(\bar{b})$
and  includes a subtuple
$(a_{i_1},\dots,a_{i_k})$
 with
 $\dim((a_{1},\dots,a_{n})/\bar{b})\geqslant \dim((a_{i_1},\dots,a_{i_k})/\bar{b})=k$. Therefore $D_{\varphi,k}(\bar{y})$ is the desired formula.
\end{proof}
\begin{fact}\label{f:2.3}\hfill
\begin{enumerate}
\item \textit{Symmetry}.
 $Y\forkindep_{X}^{\dim} Z$ if and only if  $Z\forkindep_{X}^{\dim} Y$.
\item
\textit{Transitivity.}
$Y\forkindep_{X}^{\dim}Z_{1}Z_{2}$ if and only if $Y\forkindep_{X}^{\dim}Z_{1}$ and  $Y\forkindep_{XZ_{1}}^{\dim}Z_{2}$.
\item\textit{$\acl$-Preservation.}
 $Y\forkindep_{X}^{\dim} Z$ if and only if   $\acl(Y)\forkindep_{\acl(X)}^{\dim} \acl(Z)$.
 \item\textit{$\dim$-Morley Sequences.}  Any order indiscernible sequence    $\{a_{i}:\hspace*{3mm}i\in I\}$ over $X$     is a $\dim$-Morley sequence, i.e. for any two disjoint subsequences $J_{1}$ and $J_{2}$ of $I$ we have $J_{1}\forkindep_{X}^{\dim}J_{2}$.
\end{enumerate}
\end{fact}
We will need formulas that
are ``strong'' in the sense that they
carry
enough information about $\dim$.
These are characterized  by the following Definition.
\begin{dfn}
We say that an
$\mathcal{L}$-formula $\varphi(\bar{x},\bar{y})$ is {\em strong} with respect to
the distinct tuples $\bar{d},\bar{a}$ if
$\varphi(\bar{x},\bar{a})\in \qftp_{\mathcal{L}}(\bar{d}/\bar{a})$
and for each
$\bar{d}^{'}\bar{a}^{'}$  in a sufficiently saturated model $M\models T$,
\begin{enumerate}
\item   if $M\models \varphi(\bar{d}^{'},\bar{a}^{'})$ then $\bar{d}^{'}\bar{a}^{'}$
is distinct, and
\item  if  $M\models \varphi(\bar{d}^{'},\bar{a}^{'})$ and
$\dim(\bar{d}^{'}/\bar{a}^{'})=\dim(\bar{d}/\bar{a})$ then  for each  partition
$P=(\bar{x}_{1},\bar{x}_{2})$  of $\bar{x}$, $\dim(\bar{d}^{'}_{2}/\bar{d}^{'}_{1}\bar{a}^{'})\leqslant \dim(\bar{d}_{2}/\bar{d}_{1}\bar{a})$.
\end{enumerate}
\end{dfn}
\begin{rem}\label{rem:2.1.5}
The second item  above implies,  in particular, that
$\dim(\bar{d}^{'}_{1}/\bar{a}^{'})\geqslant \dim(\bar{d}_1/\bar{a})$.
\end{rem}
The $\bigvee$-Definability
(Fact \ref{fact:1.1}) implies that
a strong formula
for any pair of distinct sequence
always exists.
We shortly give a proof of this in the following,
where for a sequence of distinct variables $\bar{x}$ we denote by $\mathcal{P}$ the set of all partitions $P =(\bar{x}_1,\bar{x}_2)$ of $\bar{x}$.
\begin{lem}\hfill
\label{existance-of-strong}
\begin{enumerate}
\item
Suppose that $\bar{d},\bar{a}$ is a pair of distinct sequences. Then
there is
a strong formula $\varphi(\bar{x},\bar{y})$ with respect to $\bar{d},\bar{a}$.
\item
If
$\varphi(\bar{x},\bar{y})$ is strong with respect to $\bar{d},\bar{a}$,
$T\models   (\theta(\bar{x},\bar{y})\to \varphi(\bar{x},\bar{y}))$ and
$\theta(\bar{x},\bar{a})\in\qftp_{\mathcal{L}}(\bar{d}/\bar{a})$  then
$\theta(\bar{x},\bar{y})$ is also strong.
\end{enumerate}
\end{lem}
\begin{proof}
We prove only item 1, and 2 would be clear.
For  each partition
$P=(\bar{x}_{1},\bar{x}_{2})$ of $\bar{x}$, by the $\bigvee$-Definability
(Fact \ref{fact:1.1})
there is a formula
$\varphi_P(\bar{x}_2,\bar{x}_1,\bar{y})$
such that
$\varphi_{P}(\bar{x}_{2},\bar{d}_{1},\bar{a})\in \qftp_{\mathcal{L}}(\bar{d}_{2}/\bar{d}_{1}\bar{a})$ and if
$M\models \varphi_{P}(\bar{d}^{'}_{2},\bar{d}^{'}_{1},\bar{a}^{'})$ then
$\dim(\bar{d}^{'}_{2}/\bar{d}^{'}_{1}\bar{a}^{'})\leqslant \dim(\bar{d}_{2}/\bar{d}_{1}\bar{a})$.
Now we simply let $\varphi(\bar{x},\bar{y})$
be the conjunction of all formulas $\varphi_P$ with the formula that states
$\bar{x}$ and $\bar{y}$ are distinct.
\end{proof}
In the next sections, we deal with
a class of bi-colored expansions of models of $T$. For this class to have
certain desired features,
we need $T$ to satisfy yet two other properties,
that is the free-amalgamation, and the
indecomposability,
 as introduced in the following.
These features will later  lead us to an axiomatizations for the ``rich" structures.\par
Recall that an $\mathcal{L}$-embedding $f:M\to N$ is {\em algebraically closed} if $\acl(f(M))=f(M)$.
\begin{dfn}
\label{free-amalgamation-property}
We say that
$T$
has the \textit{free-amalgamation property} over algebraically closed substructures  if
for each
$M_{0}, M_{1}, M_{2} \models T^{\forall}$
and each pair of algebraically closed
embeddings
$f_{1}:M_{0}\to  M_{1}$,
$f_{2}: M_{0}\to M_{2}$, there are
$ M\models T$
and
embeddings
$g_{1}: M_{1}\to M$
,
$g_{2}: M_{2}\to  M$
such that
\begin{enumerate}
\item
$g_1\circ f_1=g_2\circ f_2$, and
\item
$g_{1}(M_{1})$ and $g_{2}(M_{2})$ are free
over $g_{1}\circ f_{1}(M_{0})$ in $M$.
\end{enumerate}
We call $M$ a \textit{free amalgam} of $M_{1}$ and
$M_{2}$ over $M_{0}$, and although $M$ is not unique up to isomorphism and by abuse of notation, denote it by $M_{1}\oplus_{M_{0}} M_{2}$.
\end{dfn}
\begin{obser}\label{obs:1.7}
Let $T$ have the free-amalgamation  property and $M$ be a
 $\lambda$-saturated model of
 $T$. Suppose that
 $\Sigma(\bar{x})$ is a partial type over $X$
which has a solution
with no intersection with
 $\acl(X)$.
Then for each $Y\supseteq X$ with $|Y|<\lambda$
there is  a solution
 $\bar{d}^{'}$  to $\Sigma$ in
 $M$ such that
  $\bar{d}^{'}\cap Y=\emptyset$. Furthermore $\bar{d}^{'}$ can be chosen in such that $X {\bar{d}^{'}}$ and $Y$ are free  over $X$.
  \end{obser}
\begin{proof}
There is a solution $\bar{d}$ in
$M\oplus_{\acl(X)}M$ which satisfies the conditions required. By saturation we can find a solution $\bar{d}^{'}$ in $M$ such that  $X\cup {\bar{d}^{'}}$ and $Y$ are free  over $X$.
\end{proof}

\begin{dfn}
We call $T$ \textit{indecomposable}
 if
no
finite-dimensional algebraically closed set $X$
can be written as
a finite union $X=Y_1\ldots Y_n$
with $\dim (Y_i)<\dim(X)$ for $i\leqslant n$.
\end{dfn}
The assumption of indecomposability
is meant to provide
the following desirable property for bases.
\begin{lem}\label{lem:1.8}
Assume that $T$ is indecomposable and
$M\models T$. Let $B=\{d_1,\dots,d_m\}$ be an independent set over $A\subseteq M$. Then for
 each natural number $n$ there is a subset
$D\subseteq \acl(Ad_1,\dots,d_m)$
with
$|D|=n$
such that each $m$-element subset of $BD$
is a base for $\acl(Ad_1,\dots,d_m)$.
\end{lem}
\begin{proof}
By  assumption
$\dim(d_{1},\dots,d_{m}/A)=m$. Let
$S_{k}$ be the group of permutations of
$\{1,\dots,k\}$. Since
$T$ is indecomposable, the set
\begin{align*}
D_{m+1}:=\acl(Ad_{1},\dots,d_{m})\setminus\bigcup_{i=1}^{m}\acl(Ad_{1},\dots,\hat{d_{i}},\dots,d_{m})
\end{align*}
is non-empty. Let
$d_{m+1}\in D_{m+1}$ be arbitrary. By the exchange property,  the set
$\{d_{\sigma(1)},\dots,d_{\sigma(m)}\}$
is a base for
$\{d_{1},\dots,d_{m}, d_{m+1}\}$ over $A$ for any
$\sigma \in S_{m+1}$.
\par
Again by indecomposability of
$T$, the set
\begin{align*}
D_{m+2}:=\acl(Ad_{1},\dots,d_{m+1})\setminus\bigcup_{1\leqslant i< j\leqslant m+1}\acl(Ad_{1},\dots,\hat{d_{i}},\dots,\hat{d_{j}},\dots,d_{m+1})
\end{align*}
is non-empty. Now let
$d_{m+2}$ be an arbitrary element of
$D_{m+2}$. Again by the exchange property the set
$\{d_{\sigma(1)},\dots,d_{\sigma(m)}\}$
is base for
$\{d_{1},\dots, d_{m+2}\}$ over $A$ for any
$\sigma \in S_{m+2}$.
\par
Iterating this process we obtain the set $D=\{d_{m+1},\dots,d_{m+n}\}$ as desired.
\end{proof}
Examples of indecomposable geometric theories with free-amalgamation include
strongly minimal expansions of  divisible abelian groups,
o-minimal expansions of  ordered divisible abelian groups (ODAG),
any completion of the theory of algebraically closed valued fields, and
the theory of $p$-adic fields.
The indecomposability of the mentioned  theories indeed stems in the fact  that no vector space over an infinite field can be written as a finite  union of its proper subvector spaces.
\subsection{Bi-colored Expansions}
From here onward we assume that
$T$
is an indecomposable  geometric theory  satisfying the free-amalgamation property. Before proceeding, we need
to fix some more notation and conventions.
\paragraph{Notation and Conventions.}
We expand
$\mathcal{L}$
to
$\mathcal{L}_{p}=\mathcal{L}\cup \{p\}$
by adding a unary predicate
$p$
called the \textit{coloring predicate}.
We denote by
$\mathcal{K}$
the class of all
$\mathcal{L}_{p}$-structures
whose universe
$M$ is a model of $T^{\forall}$.
For
$X\subseteq M$,
by
$\langle X \rangle_{M}$, or $\langle X\rangle$
we denote the
$\mathcal{L}_{p}$-structure generated by
$X$
in
$M$.
In this sense  we call
$M$
\textit{finitely generated} if there is a finite set
$A\subseteq M$
such that
$M=\langle A\rangle$. For technical convenience, we take
$\emptyset$
also as a finitely generated structure in
 $ \mathcal{K}$.
For the ease of reference, we say
$x$ \textit{is colored}
if
$M\models p(x)$. We would simply
write $p(x)$ instead of $M\models p(x)$
when the meaning is clear from the context. We write
$p(\bar{x})$
instead of
$\bigwedge_{i=1}^{n}p(x_{i})$,
when $\bar{x}=(x_1,\ldots,x_n)$.
Also by
$p(X/Y)$
we mean the set of colored elements of
$X\setminus Y$.
Finally we fix a real number
$0<\alpha \leqslant 1$ for the rest of the paper. We will occasionally add the assumption that
$\alpha$ is/is not rational.
The $\mathcal{L}$ and $\mathcal{L}_p$-algebraic closures are respectively denoted by $\acl$ and $\Acl$.
\begin{dfn}
For
$M\in \mathcal{K}$
and a finite subset
$A$ of
$M$, define
\begin{align*}
\delta_{\alpha}(A)=\dim(A)-\alpha|p(A)|.
\end{align*}
More generally for
$A\subseteq_{\text{fin}}M$
and
$X\subseteq M$
we define
$\delta_{\alpha}(A/X)=\dim(A/X)-\alpha |p(A/X)|$. The function $\delta_{\alpha}$ is called  a  \textit{pre-dimension} map.
\end{dfn}
The assumption of quantifier elimination makes
$\delta_\alpha(A)$ independent from the ambient $M$.
It is also routine to check that
\begin{align*}
\delta_{\alpha}(AB/C)=\delta_{\alpha}(A/BC)+\delta_{\alpha}(B/C),
\end{align*}
and that $\delta_a$ is submodular, that is
\[\delta_{\alpha}(AB)+\delta_{\alpha}(A\cap B)\leqslant \delta_{\alpha}(A)+\delta_{\alpha}(B).\]
\paragraph{More on Notation.}
 Denote by
$\mathcal{K}_{\alpha}^{+}$
the subclass
of
$\mathcal{K}$
consisting of the
$\mathcal{L}_{p}$-structures
$M\in\mathcal{K}_{\alpha}^{+}$
such that
$\delta_{\alpha}(A)\geqslant 0$
for all
$A\subseteq_{\text{fin}}M$.
When $A$ is a subset of a structure
$M$ in the $\mathcal{K}^+_\alpha$,
we simply say that $A$ is in $\mathcal{K}^+_\alpha$. Furthermore in this situation for $N\in \mathcal{K}_{\alpha}^{+}$
we say that
$f: A \to N$
is an $\mathcal{L}_{p}$-embedding if  $\qftp_{\mathcal{L}} (f(A))=\qftp_{\mathcal{L}}(A)$ and $f(A)$   has the same color in $N$ as $A$ in $M$.
Finally, to avoid repeating
 the assumption that a structure is in
$\mathcal{K}^+_\alpha$ in each statement,
we simply assume that each structure and set in the following
 belongs to
$\mathcal{K}^+_\alpha$ unless we emphasize otherwise.
\par
It is clear by our previous conventions that $\emptyset\in \mathcal{K}^+_\alpha$. Also note that by  $\bigvee$-Definability  (Fact
\ref{fact:1.1}),
$\mathcal{K}_{\alpha}^{+}$
is axiomatizable in
$\mathcal{L}_{p}$ by the theory
consisting of the $\mathcal{L}_{p}$-sentences
\begin{align*}
 \neg \exists x_1,\ldots, x_n\big(\psi_l(x_1,\dots,x_n)\wedge \bigvee _{\substack{Y\subseteq \{x_1,\dots,x_n\} \\
|Y|\geq k}} p(Y)\big)
\end{align*}
 where
 $ \psi_l $ asserts that the dimension of $(x_1,\ldots,x_n)$ is bounded by $l$,
and  $l<\alpha k$.
\par
Notice that since $\alpha>0$ all algebraic points over $\emptyset$ are non-colored.
\par
We need to review some basic concepts
concerning the properties of the
pre-dimension function $\delta_{\alpha}$.
\begin{dfn}\hfill
\label{212}
\begin{enumerate}
\item
For
$A\subseteq_{\text{fin}}M$
and
$X\subseteq M$
we say that
$A$
is
closed in
$X$
and write
$A\leqslant_{\alpha}X$
if
$A\subseteq X$
and
$\delta_{\alpha}(B/A)\geqslant 0$
for all
$A\subseteq B\subseteq_{\text{fin}}X$.
\item
For
$X, Y$
arbitrary subsets of
$M$,
we say $X$ is closed in $Y$ and write
$X\leqslant _{\alpha} Y$
if
$X\subseteq Y$
and
$\delta_{\alpha}(A/X) \geqslant 0$
for all
$A\subseteq_{\text{fin}}Y$.
\item
For
structures $M,N$
we
say that $M$ is closed in $N$ and
 write
$M\leqslant_{\alpha} N$ if
$M$ is a substructure of $N$ and
$M\leqslant_\alpha N$ as sets in the sense above.
\end{enumerate}
\end{dfn}
For simplicity, we omit the subscript
$\alpha$
from
$\leqslant_{\alpha}$
and
$\delta_{\alpha}$.
It is easy to see that
if
$X\leqslant M$,
then
$\neg p(x)$
for all
$x\in \acl_{M}(X)\setminus X$. Also in this case, both
$ \langle X\rangle_M$
and
$\acl_{M}(X)$
are closed in
$M$.
\begin{rem}\label{mol:5}
$(\mathcal{K}_{\alpha}^{+}, \leqslant)$
is a so-called smooth class, that is for all
$M,M_1,M_2,X$,
\begin{enumerate}
\item
$\emptyset,M\leqslant M$.
\item
If
$M\leqslant M_{1}$
and
$M_{1}\leqslant M_{2}$, then
$M\leqslant M_{2}$.
\item
If
$M\leqslant M_2$ then
$M\leqslant M_{1}$ for
all
$M\subseteq M_{1}\subseteq M_2$.
\item
if $M\leqslant M_1$ then
$M\cap X\leqslant X$
for all
$X\subseteq M_1$.
\end{enumerate}
\end{rem}
Note also that
items 2 and 4 above easily imply that
if $M_1,M_2\leqslant M$
 then
$M_1\cap M_{2}\leqslant M$.
\begin{dfn}
Assuming that
$X,A,B\subseteq M$,
\begin{enumerate}
\item
The
\textit{closure} of
$X$ in
$M$, denoted by $\cl_{M}(X)$,
is
the smallest subset $Y$ of $M$
containing
$X$ such that $Y\leqslant M$.
\item
 We say that
 $B$ is an \textit{intrinsic extension} of
 $A$  and
 write
 $A\leqslant_{i} B$ if
 $A\subseteq B$
 but there is no
 $A^{'}\neq B$
with
$A\subseteq  A^{'}\leqslant B$,
equivalently, $\delta(B)\lneqq\delta(A')$
for all $A\subseteq A'\varsubsetneqq B$.
\item A pair
$(A, B)$
is called
\textit{minimal} if
$A\subseteq B$,
$A\nleqslant B$
but
$A\leqslant C$
for all
$A\subseteq C\subsetneqq B$.
\end{enumerate}
It is clear that if $(A,B)$ is a minimal pair, then
$ A\leqslant_i B $. Moreover if $ B $ is an intrinsic extension of
$ A $ then it is possible to find a tower
$ B_0=A\subseteq B_1\subseteq \ldots\subseteq B_n=B $ where each
$ (B_i,B_{i+1}) $ is minimal.
\end{dfn}
The following statements are well-known facts about basic properties of $\cl_M$.
\begin{fact}\hfill
\begin{enumerate}
\item
 $\cl_{M}(X)$
is the intersection of all closed subsets of
$ M$ that contain $X$.
\item
 $\cl_{M}(X)=\bigcup_{A\subseteq  X} \cl_{M}(A)$.

\item   $\cl_M(A)=\bigcup \{B\subseteq M: A\leqslant_i B\}$.
\end{enumerate}
\end{fact}
When $M$ is clear from the context,  we write
$\cl(X)$ instead of $\cl_M(X)$.
\begin{dfn}
For
$A\subseteq  M$ and natural number $n$,
we define $\cl^{n}_{M}(A)$ as the union of all $B\subseteq M$  such that $A\leqslant_i B$ and $|B-A|<n$.
\end{dfn}
\begin{rem}\label{rem:2.17}
Let $ B, M\in \mathcal{K}_{\alpha}^+$ and $A\subseteq M$ . Then
 \begin{enumerate}
\item
  There is
$f:\mathbb{N} \times  \mathbb{N} \to \mathbb{N}$ such that for any embedding $g:B\to M$
the size of
$\cl^{n}_{M}(g(B))$ is  bounded by
$f(|B|,n)$.
\item
$\cl_{M}(A)\subseteq \Acl_{M}(A)$.
\item
 When $\alpha$ is rational, then $\cl_M(A)$ is finite; hence $\cl_M(A)=\cl_M^n(A)$ for some $n$.
\item
  By item 1,  $\cl_M(A)$ is countable when $\alpha$ is irrational.

\end{enumerate}
\end{rem}
 The following definition  singles out two different types of  closed extensions. This distinction  will be highlighted in the axiomatization of rich structures given in section $3.1$.
\begin{lem}
Suppose that $A\subseteq M$ is finite.
Then there is a finite set $B$
containing $A$ such that
$\langle A\rangle=\langle B\rangle$ and
and $B\leqslant \langle B\rangle=\langle A\rangle$.
\end{lem}
\begin{proof}
If $c_{1},\ldots , c_{k}\in \langle A\rangle\setminus A$  are colored elements then
\begin{align*}
0\leqslant \delta(Ac_{1},\ldots,c_{k})=\delta(A)-k\alpha.
\end{align*} Hence
if $C\subseteq\langle A\rangle$ and each elements of $C$ is colored then
\begin{align*}
|C\setminus A|\leqslant \delta(A)/\alpha\leqslant |A|/\alpha.
\end{align*}
So let $B=A \cup
p(\langle A\rangle)$.
\end{proof}
\paragraph*{Remark and Convention.}
For $B$ as in the above lemma,
each $x\in \langle B\rangle -B$ is non-colored. Hence we make the convention that
whenever we write $\langle A\rangle$
we assume that $A\leqslant \langle A\rangle$. Notice that by this convention for  $A\subseteq B$ we have that   $\langle A\rangle \leqslant \langle B\rangle$ if and only if $A\leqslant B$.
\begin{dfn}
 Assume that
 $\langle A\rangle\subseteq \langle B\rangle$ and
 $ A\leqslant B$ . Call
$B$
\begin{enumerate}
\item
\textit{algebraic} over
$A$ if
$\dim(b/A)=0$
for all
$b\in B\setminus A$, and
\item
\textit{transcendental}
over
$A$
if
$\dim(b/A)=1$
for all
$b\in B\setminus A$.
\end{enumerate}
\end{dfn}
\begin{rem}\label{rem:2.19}
It can be easily seen that if
$A\leqslant B\subseteq \langle B \rangle$
then there exists
$B_{1}$ such that
$A\leqslant B_{1}\leqslant B$
with
$B_{1}$ algebraic over
$A$ and
$B$
transcendental over
$B_{1}$. Furthermore, if $ B$ is algebraic over $A$  and  $f:\langle B \rangle \to M$ is an $\mathcal{L}_p$-embedding in some $M\in\mathcal{K}_{\alpha}^{+}$ then $ \cl_M(f(B))=\cl_M(f(A))\oplus_{f(A)}f(B) $. So, in particular, for each natural $n$ we have $\cl^n_{M}(f(B))=\cl^n_{M}(f(A))\oplus_{f(A)}f(B)$.
\end{rem}
\section{  Rich Structures and their Axiomatization}
\subsection{Rich Structures}
We have already defined what we mean by $T$
having
the free-amalgamation property (see Definition \ref{free-amalgamation-property}).
Along the same lines,
we say that
$(\mathcal{K}_{\alpha}^{+},\leqslant )$ has the free-amalgamation property (defined more precisely after the following definition)
if
the conditions in Definition
\ref{free-amalgamation-property} hold
with $M_i, M$ structures in the class
$\mathcal{K}^+_\alpha$,
and
$f_i,g_i$ strong embeddings in the sense of the following definition.
\begin{dfn}
\label{strong-embdedding}
By
a \textit{strong embedding}
$f: M\to  N$
we mean an
$\mathcal{L}_{p}$-embedding
$f$
such that
$f(M)\leqslant N$.
%If  $A,B$ are sets, then
%we similarly define a strong embedding
%$f: A\to B$ as
%an embedding  $f:\langle A\rangle\to \langle B\rangle$ such that
%$f(\langle A\rangle)\leqslant \langle B\rangle$.
\end{dfn}
So by
the class \mbox{$(\mathcal{K}^+_\alpha,\leqslant)$} having the amalgamation property we mean that
for each
$M_{0}, M_{1}, M_{2}$
and each pair of strong
embeddings
$f_{1}:M_{0}\to  M_{1}$,
$f_{2}: M_{0}\to  M_{2}$, there are
$ M$
and
strong
embeddings
$g_{1}: M_{1}\to  M$,
$g_{2}: M_{2}\to M$
such that
$g_1\circ f_1 =g_2\circ f_2$.
\begin{lem}\label{lem:17}
The class \mbox{$(\mathcal{K}^+_\alpha,\leqslant)$} has the amalgamation property. Moreover, if  both $f_1,f_2$ are algebraically-closed then we can  choose $ M\in \mathcal{K}^+_\alpha $ such that
 $g_{1}(M_{1})$ and  $g_{2}(M_{2})$ are free over $g_{1}\circ f_{1}(M_{0})$ in $M$.

\end{lem}
\begin{proof}
Assume that
$\mathcal{L}_{p}$-structures
$M_{0}, M_{1}, M_{2}$ in $\mathcal{K}_{\alpha}^{+}$
and  the strong
embeddings
$f_{1}$,
$f_{2}$
 are as above.
 \par
We can assume, without loss of generality that
$M_1,M_2$ are models of $T$.
This can be obtained
by   $\mathcal{L}$-embedding of each $M_1$ and $M_2$ into a model of $T$ and a suitable extension of the coloring.
\par
 We may also assume that $f_{1}$,
$f_{2}$ are algebraically-closed strong embeddings. This is because
  $\acl_{M_1}(M_0)$ and $\acl_{M_2}(M_0)$
are
$\mathcal{L}_p$-isomorphic over $M_0$,
as
both $f_1$ and $f_2$ are strong (see the explanation after Definition \ref*{212}).
\par
Now since
$T$
has the free-amalgamation property,  there are a
model
$M$
and $\mathcal{L}$-embeddings
$g_{i}: M_{i} \to M$,
such that $g_1\circ f_1=g_2\circ f_2$.
Defining
$p(M)=g_{1}(p(M_{1}))\cup g_{2}(p(M_{2}))$, it is clear that
$M\in \mathcal{K}^+_\alpha$.
\par
It remains only to prove that
$g_{1},g_{2}$
are  strong
embeddings.
For notational simplicity
we denote again by $M_1,M_2$ and $M_0$
their respective images in $M$, that is
$g_{1}(M_{1})$, $g_{2}(M_{2})$
and
$g_{1}\circ f_{1}(M_{0})$.
In the following we show that $M_1\leqslant M$ (and $M_2 \leqslant M$
 follows similarly).
\par
 Let
$A\subseteq  M\setminus M_{1}$ be arbitrary.  Writing
$A=BC$
where
$B\subseteq {M}_{2}$,
$C\subseteq M\setminus {M}_{2}$
we have
\begin{align*}
\delta(A/{M}_{1})=\delta(C/{M}_{1})+\delta(B/{M}_{1}C).
\end{align*}
Now $\delta(C/{M}_1)\geqslant 0$ because
all elements of $C$ are non-colored.
Also $\delta(B/{M}_1C)=\delta(B/{M_0}C)\geqslant 0$, because
 ${M}_1\forkindep^{\dim}_{{M}_0}{M}_2$.
\par
\end{proof}
As we have let
$\emptyset$ in $\mathcal{K}^+_\alpha$,
the amalgamation property implies
the joint embedding property. It is also clear that
if $M$ is in  $\mathcal{K}_{\alpha}^{+}$
then so are all its substructures.
In this sense,
we call
$(\mathcal{K}_{\alpha}^{+},\leqslant )$
a \textit{Fra\"{i}ss\'{e} class}, and
it is natural to look for  its Fra\"{i}ss\'{e} limit, or the \textit{rich}
structures in the sense below.
\begin{dfn}\label{def:11}
An
$\mathcal{L}_{p}$-structure
$M$
in
$\mathcal{K}_{\alpha}^{+}$
is called
$\lambda$-rich,
for $\lambda\geqslant \aleph_0$ a cardinal,
  if
\begin{enumerate}
\item
$M\models T$,
\item
If
$M_1\leqslant M_2$
are
generated by a set of cardinality
$<\lambda$,
then each strong embedding
$f: M_{1}\to M$
extends to
a strong embedding
$g: M_{2}\to M$.
\end{enumerate}
\end{dfn}
The Amalgamation and Joint Embedding properties  together with the fact
that $\mathcal{K}^+_\alpha$ is closed
under the
union of $\leqslant$-chains, imply
that the $\lambda$-rich structures exist.
Indeed a union of $\leqslant$-chains of models of $T$, and
the
quantifier-elimination (for this union to be  a model) give such a structure.
We prefer to state this also as the following fact.
\begin{fact}
$(\mathcal{K}_{\alpha}^{+},\leqslant )$ contains
$\lambda$-rich structures,
for all
$\lambda\geqslant \aleph_{0}$.
\end{fact}
\begin{lem}
\label{rich-saturated-1}
Each $\lambda$-rich structure is
$\lambda$-saturated as a model of $T$.
\end{lem}
\begin{proof}
Let
$M$
be
$\lambda$-rich.
Assume that
$\Sigma(x)$ is a partial  $1$-type over a set
$X\subseteq {M}$,
where $|X|<\lambda$ and
without loss of generality  assume that
$X$  is closed in
${M}$.
 Let
 $d\not \in  {M}$ be a solution of
$\Sigma(x)$ in
 an  $\mathcal{L}$-elementary  extension
 ${N}$ of
${M}$.
Extend the coloring of ${M}$ to ${N}$
by letting $\neg p(x)$ for each $x\in {N}-{M}$, so that now ${N}\in \mathcal{K}_\alpha^+$.
Observe that
$\cl_{M}(X)\leqslant {M},{N}$ and hence we keep the notation
$\cl(X)$.
It is clear that
 $\langle \cl(X)d\rangle\in \mathcal{K}^+_{\alpha}$,
$\langle \cl( X)\rangle\leqslant \langle \cl(X)d\rangle$
and $\langle \cl(X)\rangle\leqslant {M}$.
Now since
$M$  is $\lambda$-rich, there is  a strong embedding
$f:\langle \cl(X)d\rangle \to {M}$, and hence
$f(d)$
is the solution of
$\Sigma(x)$ in
${M}$.
\end{proof}

\par
It is  worth noting that
in the literature,
a
$\lambda$-rich structure is sometimes called
$\lambda$-ultra-homogeneous.
Moreover, it is clear from the definition
(letting
$M_{1}=\emptyset$) that for
each
$M_{2}$
as above,
there is a  strong embedding
$g: M_{2}\to  M$.
This property of $M$ is called
$\lambda$-universality.
\par
We are now set to provide an axiomatization
for the rich structures.
\subsection{Axiomatization}
Below for the $\mathcal{L}$-formula $\varphi(\bar{x}, \bar{y})$ and  a natural number $k$, we let $D_{\varphi, k}(\bar{y})$  and $d_{\varphi, k}(\bar{y})$   be
the formulas introduced in Lemma
\ref{defacl} and Fact \ref{fact:1.1} respectively.
\begin{dfn}
\label{def-of-t-alpha}
Let
$\mathbb{T}_{\alpha}$
be an
$\mathcal{L}_{p}$-theory whose models
$\mathbb{M}$
satisfy the following.
\begin{enumerate}
\item
$\mathbb{M}\models T$
\item
$\mathbb{M}\in \mathcal{K}_{\alpha}^{+}$ (that is
$\delta(A)\geqslant 0$
for all
$A\subseteq_{\text{fin}} \mathbb{M}$).
\item
For each $\bar{a}\leqslant \bar{a}\bar{d}$,  where $\bar{d}=(d_{1},...,d_{n})$ is transcendental over $\bar{a}$, $\dim(\bar{d}/\bar{a})=k$, and
$\varphi(\bar{x},\bar{y}) $ is strong with respect to
$\bar{d},\bar{a}$,
\begin{align*}
\mathbb{M}\models &\forall \bar{y}\hspace*{3mm}\Big{[}D_{\varphi,k}(\bar{y})\wedge d_{\varphi,k}(\bar{y})
\to \exists \bar{z}\hspace*{2mm}\big{(}\varphi(\bar{z},\bar{y})\wedge
 \bigwedge_{p(d_i)}p(z_{i})\wedge \bigwedge_{\neg p(d_i)}\neg p(z_{i})\hspace*{2mm} \big{)} \Big{]}\hspace*{2mm}
\end{align*}
\end{enumerate}
\end{dfn}
The items above can be expressed
as first-order axiom-schemes (for the second item
this is pointed out before
Definition \ref{212}).
In the rest of this section we aim to establish the following theorem.
\begin{thm*}
$\mathbb{T}_{\alpha}$
is a complete theory that axiomatizes $\omega_1$-rich structures of
$\mathcal{K}_{\alpha}^{+}$.
\end{thm*}
Towards a proof of the theorem above, what we can prove in the rest of this subsection with
the information so far,
is that $\omega_1$-rich
structures in general, and $\omega$-rich
structures when $\alpha$ is rational,
are a model of $\mathbb{T}_\alpha$.
Proving that models of $\mathbb{T}_\alpha$ are rich and therefore giving
the full proof of the theorem will be
possible only in Section 3.3, after some more auxiliary lemmas in the next section.
\begin{lem}
Any
$\omega_{1}$-rich structure
$\mathbb{M}$
 is a model of
$\mathbb{T}_{\alpha}$ (and hence
$\mathbb{T}_{\alpha}$ is consistent).
\end{lem}
\begin{proof}
By Lemma \ref{rich-saturated-1},
$\mathbb{M}$ is $\omega_1$-saturated as
a model of $T$.
Now we only need to prove that
item 3 in Definition
\ref{def-of-t-alpha} holds.
Let $\bar{a},\bar{d}$ and
$\varphi$ be as in there,
and $\bar {a}^{'}$ in $\mathbb{M}$ be such that $\mathbb{M}\models D_{\varphi,k}(\bar{a}^{'})\wedge d_{\varphi,k}(\bar{a}^{'})$.
So there is
$\bar{d}^{'}\in \mathbb{M}^{|\bar{d}|}$
such that
$\mathbb{M}\models \varphi(\bar{d}^{'},\bar{a}^{'})$, $\bar{d}^{'}\cap \acl(\bar{a}^{'})=\emptyset$,
and
$\dim(\bar{d}^{'}/\bar{a}^{'})=k$.   Since $\langle
 \cl(\bar{a}^{'})\rangle $ is countable, by
$\omega_{1}$-saturation of
$\mathbb{M}$ and  Observation
\ref{obs:1.7}
we may assume that
$\langle \cl(\bar{a}^{'})\rangle \cap \bar{d}^{'}=\emptyset$
and, $ \bar{d}' $ and $ \langle \cl(\bar{a}') \rangle $
are free over $ \bar{a} '$.
Consider the
$ \mathcal{L} $-structure   generated by
$ \cl(\bar{a}')\bar{d}' $ in $ \mathbb{M} $ and make it into
the $ \mathcal{L}_p $-structure $ N $ by coloring
$ \cl(\bar{a}')\bar{d}' $  with the same colors as $\cl(\bar{a}) \bar{d} $
and leaving the rest of the elements non-colored.
It is clear that $\langle\cl(\bar{a}')\rangle \leqslant \mathbb{M}$.
We claim that also $\langle\cl(\bar{a}')\rangle \leqslant {N}$.
If this claim is proved then since
$\mathbb{M}$  is $\omega_{1}$-rich, there is a strong embedding
$f:N\to \mathbb{M}$ which  fixes
$\langle\cl(\bar{a}^{'})\rangle $  pointwise. If we put
$f(\bar{d}^{'})=\bar{e}$,
then
\begin{align*}
\mathbb{M}\models \big{(}\varphi(\bar{e},\bar{a}^{'})\wedge \bigwedge _{p(d_i)}p(e_{i})\wedge \bigwedge_{\neg p(d_i)}\neg p(e_i)\big{)}.
\end{align*}
Now to prove the claim,  let $ \bar{d}'_1\subseteq \bar{d} '$ and
$  \bar{e}\subseteq N\setminus \cl (\bar{a}^{'})\bar{d}'$. Notice that
$\dim(\bar{d}'_1/\langle \cl(\bar{a}^{'})\rangle)=\dim(\bar{d}^{'}_{1}/\bar{a}^{'})$ since $\bar{d}\forkindep^{\dim} _{\cl(\bar{a}^{'})}\bar{a}^{'}$. Moreover as $\varphi$ is strong, by \Cref{rem:2.1.5}  we have  $\dim(\bar{d}^{'}_{1}/\bar{a}^{'})\geqslant \dim(\bar{d}_{1}/\bar{a})$ . Therefore
\begin{align*}
\delta(\bar{d}'_1\bar{e}/\langle \cl(\bar{a}^{'})\rangle)&=
\dim(\bar{d}'_1\bar{e}/\langle \cl(\bar{a}^{'})\rangle)-\alpha |p(\bar{d}'_1\bar{e})|\\
&=\dim(\bar{d}'_1/\langle \cl(\bar{a}^{'})\rangle)-\alpha |p(\bar{d}'_1)|\\
&=\dim(\bar{d}^{'}_{1}/\bar{a}^{'})-\alpha |p(\bar{d}'_1)|\\
&\geqslant \dim(\bar{d}_{1}/\bar{a})-\alpha |p(\bar{d}_1)|\geqslant 0.
\end{align*}
\end{proof}
It follows from the proof above that
in case $\alpha$ is rational, $\omega$-rich structures are a model of
$\mathbb{T}_{\alpha}$.
\subsection{Auxiliary Lemmas}
The following lemmas and their proofs are
inspired by
\cite{L07}, Lemmas $3.1$ and  $4.1$.
When $A\leqslant B$ and $B$
is transcendental over $A$, these lemmas
provide us with a way to find a set
$D$  containing
$B$ (in both cases of rational and irrational $\alpha$) whose  pre-dimension   is  as much  close from below   to $\delta(B)$ as we wish,
yet $A\leqslant D$ and $D$ is transcendental over $A$.
\par
We will in particular utilize Lemmas
\ref{lem:1.22},
\ref{lem:1.25}, and \ref{lem:3.3.4}
 to show that any $\omega_1$-saturated model of
$\mathbb{T}_{\alpha}$, for an irrational
$\alpha$ is in fact semi-generic (Theorem \ref{lem:3.14}).
Similarly we will make use of Lemmas
\ref{lem:3.3.5} and \ref{lem:3.3.6} to prove that any $\omega$-saturated
model of $\mathbb{T}_\alpha$,
for a rational $\alpha$ is $\omega$-rich
(Theorem \ref{th:3.4.6}).
\begin{lem}\label{bound}
For each $n\in \omega$ there exists an $\epsilon_n>0$ such that for each $A\subseteq B$, if $|B\setminus A|<n$ and $\delta(B/A)<0$ then $\delta(B/A)<-\epsilon_n$.
\end{lem}
\begin{proof}
The set  $V_n=\{\delta(B/A):\  A\subseteq B,\  |B\setminus A|<n\}$ is finite. Therefore there exists $\epsilon_n>0$, depending on $n$, which keeps every negative value in $V_n$ bounded away from $0$, as required by the lemma.
\end{proof}
Recall that by  the finite version of the $\Delta$-system lemma \cite{J11},
if $\{B_i\}_{i\in \omega}$ is a countably-infinite   family of
finite sets of  bounded size,
then there is an arbitrarily-large
finite subfamily of
$B_i$'s  with the same mutual intersection.
The following lemma shows that
when the $B_i$ are subsets of a structure in $ \mathcal{K}^+_\alpha $, then it is possible that the intersection in question is closed in each $B_i$.
\begin{lem}\label{lem:1.22}
For  $k\geqslant 1$, let
$\{B_{i}:\hspace*{2mm} i\in \omega\}$
be a family of
$k$-element subsets of
a structure $M$. Then
for each natural number
$n$
there are  a
subset
$\Omega$
of
$\omega $ with $|\Omega|\geqslant n$
and
$A\subseteq_{\text{fin}}M$
such that
\begin{enumerate}
\item
$A\leqslant B_{i}$, for each
$i\in \Omega$ and
\item
$B_{i}\cap B_{j}=A$, for  $i\neq j\in \Omega $.
\end{enumerate}
\end{lem}
\begin{proof}
Let  $\epsilon_k$ be as in Lemma \ref{bound}. By the finite
version of $\Delta$-system lemma, for each natural number
$n$
there are a finite subset
$\Omega' \subset \omega$ with
$|\Omega'|\geqslant n +\lfloor k/\epsilon_k\rfloor $ and a finite subset $A$ of $M$ such that
$B_{i}\cap B_{j}=A$ for each
$i\neq j \in \Omega'$.
But  we need a further step to show that there are at most $  \lfloor{\frac{k}{\epsilon_k}}\rfloor  $-many $B_i$ in which $A$ is not closed.
\par
 Put
$Z=\{i\in \Omega':\hspace*{2mm}A\nleqslant B_{i}\}=\{{i_1},\ldots, {i_m}\}$. We show that
$m\leqslant \lfloor \frac{k}{\epsilon_k}\rfloor$. For $i\in Z$ since $A\nleqslant B_i$, there is $C_i$ with $A\subsetneq C_i \subseteq B_i$ and $\delta(C_i/A)<0$.
Put $C=\bigcup_{i\in Z} C_i$. Then
\begin{align*}
\delta(C/A)=\delta(C_{i_m}/AC_{i_1},...,C_{i_{m-1}})+...+\delta(C_{i_2}/AC_{i_1})+\delta(C_{i_1}/A)
&\leqslant -m\epsilon_k.
\end{align*}
So
$0\leqslant\delta(C)\leqslant -m\epsilon_k+\delta(A)\leqslant -m\epsilon_k+k$. Therefore
$m\leqslant \lfloor k/\epsilon_k\rfloor$.
\end{proof}
\begin{lem}
\label{lem:1.25}
($\alpha$ irrational) Assume that
$A\leqslant B$
and $B$ is transcendental over $A$.
Then for each $\epsilon>0$,
there
is a
finite set
$D$
containing $B$ and transcendental over $ A$
such that
\begin{enumerate}
\item
$-\epsilon<\delta(D/B)\leqslant 0 $,
\item
$\delta (B)\leqslant \delta(D')$
for all
$B\subseteq  D^{'}\subsetneqq D$,
\item
$A\leqslant D$.
\end{enumerate}
\end{lem}
%For the proof of the lemma  above we need to distinguish between the cases of irrational and rational $\alpha$. This will as well be the case for some other lemmas and theorems
%in the rest.
%\paragraph{Proof of Lemma \ref{lem:1.25}
%for the case of an irrational $\alpha$.}
\begin{proof}
Without loss of generality we may assume that
$0<\epsilon<\alpha$ and $\delta(B/A)>\epsilon$. Note that  the set $\{m-\alpha n:\ m,n\in \mathbb{N}\}$ is dense  in $\mathbb{R}$ (by  Dirichlet's rational approximation theorem).
 Hence there are natural numbers $1\leqslant s<k$
such that
$s/k<\alpha<(s+\epsilon )/k$.
\par
 By Lemma
\ref{lem:1.8} and richness there is a set
$D=Bd_{1},...,d_{k}$
such that
$p(d_{i})$ for all $i\leqslant k$
and
each $s$-element subset
of $\{d_1,\ldots,d_k\}$
together with $B$
is a base for
$D$.
We claim such $D$ satisfies the conditions required by the lemma.
\par
For the first condition note that
$-\epsilon<\delta(D/B)=s-k\alpha<0$.
For the second condition, let
 $B\subseteq D'\subseteq D$ and
$|D'-B|=l$.
If $l\leqslant s$ then $\delta(D'/B)=l-l\alpha>0$ because $\alpha<1$. Also if $s<l<k$ then
$\delta(D'/B)=s-l\alpha>s-\alpha k>0$.
 For the third condition, note that
 $A\leqslant E$  for any proper subset
 $E$ of
 $D$ containing
 $A$.
It only remains to show that
$\delta(D/A)>0$. But this is also easy to see, since
 $\delta(D/A)=\delta(D/B)+\delta(B/A)$
where
 $\epsilon<\delta(B/A)$ and $\delta(D/B)>-\epsilon$.
Finally each
 $d_{i}$ is transcendental over
 $B$, and hence over
 $A$.
 \end{proof}
%\paragraph{Proof of Lemma \ref{lem:1.25}
%for the case of a rational $\alpha$.}
%Let $s,k$ be  two natural number  such that $s-\alpha k=0$. Now as in the case of irrational $\alpha$, we take $D=Bd_{1},...,d_{k}$
%such that
%$p(d_{i})$ for all $i\leqslant k$
%and
%each $s$-element subset
%of $\{d_1,\ldots,d_k\}$
%together with $B$
%is a basis for
%$D$. Then $D$ satisfies the stronger condition $\delta(D/B)=0$ as well as the other conditions, as required.\qed
\begin{lem}\label{lem:3.3.4}
($\alpha$ irrational) Assume that
$A\leqslant B$  and $B$ is transcendental over $A$. Then for each $\mu>0$  and  natural number $n$  there is $ D^{*}$ containing $B$  such that
\begin{enumerate}
\item
$A\leqslant D^{*}$ and $D^{*}$ is transcendental over $A$,
\item
$\delta(D^{*}/A)<\mu$,
\item
$B\leqslant C$,
for each $C$ with $B\subseteq C\subseteq D^{*}$ and $|C\setminus B|<n$.

\end{enumerate}
\end{lem}
\begin{proof}
Consider a real number
$0<\lambda<\min \{\delta(B/A), \frac{1}{n}\epsilon_n\}$, where $\epsilon_{n}$ is as   in \Cref{bound}.
By \Cref{lem:1.25} there exists $D$ such that $-\lambda<\delta(D/B)<0$, $A\leqslant D$ and $D$ is transcendental  over $A$. Let $\gamma=-\delta(D/B)$, so $0<\gamma <\lambda$.
 Also let $k\geqslant 1$ be the natural number such that
\begin{align*}
k\gamma \leqslant \delta(B/A) <(k+1)\gamma.
\end{align*}
Let
 $D^{*}=D_1\oplus_B \ldots\oplus_B D_k$
 be obtained as the union of
 $k$ copies of $D$ (with the same colors) free from one another over $B$.
% (see Notation ...).
 We  first show that $A\leqslant D^{*}$ (hence $D^{*} \in \mathcal{K}_{\alpha}^{+}$) and $\delta(D^{*}/A)<\mu$.
\par
To show that $A\leqslant D^{*}$,
we need to show that
$\delta(C/A)\geqslant 0$ for all  $C$ with $A\subseteq C \subseteq D^{*}$. So let $B_{0}=C\cap B$ and
$C_{i}=C\cap D_{i}$ for each
$i<k$. There are two cases to consider:
\par
\noindent
\textbf{Case 1.}
$B_{0}=B$. \par
By \Cref{lem:1.25}, $\delta(D^{'}/B)> -\gamma$ for all $D^{'}$ that
$B\subseteq D^{'}\subseteq D$. So
$\delta(C_{i}/B)> -\gamma$ for each $i<k$. Therefore
\begin{align*}
\delta(C)=\delta(B)+\sum_{i<k}\delta(C_{i}/B)> \delta(B)-k\gamma \geqslant \delta(A).
\end{align*}
\textbf{Case 2.}
$B_{0}\neq B$.\par
Since
$A\leqslant B$, it follows that  $\delta(B_{0}/A)\geqslant 0$. Furthermore each  $C_{i}$ is proper subset of $D_{i}$ which means
$\delta(C_{i}/B_{0})\geqslant 0$ for all $i<k$ by \Cref{lem:1.25}. Hence
\begin{align*}
\delta(C/A)=\sum_{i<k}\delta(C_{i}/B_{0})+\delta(B_{0}/A)\geqslant 0.
\end{align*}
So in general
$A\leqslant D^{*}$ and hence
$D^{*}\in \mathcal{K}_{\alpha}^{+}$.
\par
Now it is easy to see that
$\delta(D^{*}/A)<\mu$, because
\begin{align*}
\delta(D^{*}/A)= \delta(D^{*}/B)+\delta(B/A)\leqslant \delta(B/A)-k\gamma <\gamma <\lambda <\mu.
\end{align*}
Finally we prove
that item 3 in the statement of the lemma is fulfilled.
 Let $C$ be subset of $D^{*}$ containing $B$  with $|C\setminus B|<n$.
% we show that
%$\delta(C/B)\geqslant -n\lambda$.
Let
$C_{i}=C\cap D_{i}$ for each
$i<k$. Since $|C\setminus B|<n$ it follows that $|\{i;\hspace*{2mm}C_{i}\neq B\}|<n$. Hence
\begin{align*}
\delta(C)=\delta(B)+\sum_{i<k}\delta(C_{i}/B)\geqslant \delta(B)-n\lambda.
\end{align*}
So
$\delta(C/B)\geqslant -n\lambda\geqslant -\epsilon_{n}$. But \Cref{bound} implies that
$\delta(C/B)\geqslant 0$.
\end{proof}
\begin{lem}\label{lem:3.3.5}
($\alpha$ rational)
Assume that
$\alpha=\frac{m}{n}$ with $m, n$ co-prime and $0<m<n$.
Let $A\leqslant B$, $B$ be  transcendental over
$A$,
$\delta(B/A)>0$ and $t$ be  a natural number. Then there is
$D$ containing $B$ such that
\begin{enumerate}
\item
$A\leqslant D$ and $D$ is transcendental over $A$.
\item
$(B,D)$ is a minimal  pair, $|D\setminus B|>t$, and
$\delta(D/B)=-\frac{1}{n}$.
\end{enumerate}
\end{lem}
\begin{proof}
Similar to the proof of \Cref{lem:1.25} we  introduce suitable
$s<k$   and
 $D=Bd_{1},...,d_{k}$
such that
$p(d_{i})$ for all $i\leqslant k$,
and
each of the $s$-element subsets of $\{d_1,\ldots,d_k\}$
together with $B$ forms
 a basis for
$D$.
\par
Let
$1\leqslant k^{'},s^{'}<n^{t+1}$
  such that
$m^{t+1}k^{'}-1=s^{'}n^{t+1}$ witness the fact that
 $(m^{t+1},n^{t+1})=1$.  Put
$s^{'}n^{t}=s$
and
$k^{'}m^{t}=k$.  Then $\delta(D/B)=s-\frac{m}{n}k=-\frac{1}{n}$.
\par
The other requirements on $D$ can be verified in a similar way to
 the proof of \Cref{lem:1.25}.
\end{proof}
\begin{lem}\label{lem:3.3.6}
($\alpha$ rational)
Under  the assumptions of the previous lemma,
there is a
 transcendental extension $D^{*}$ of $A$ such that
$\delta(D^{*}/A)=0$ and
 for each $D^{'}$ containing $B$ with
$|D^{'}\setminus B|<t$ we have
$B\leqslant D^{'}$.
\end{lem}
\begin{proof}
Suppose that $\delta(B/A)=\frac{p}{n}$.
Similar to  the proof of \Cref{lem:3.3.4} we consider $D^{*}$  to be a
disjoint free
union of  $p$-copies of $D$ over $A$ and show that $D^{*}$ has the desired properties.
\end{proof}
\subsection{Completeness}
We now turn to the theory
$\mathbb{T}_\alpha$
and showing that it is
 complete.
\subsubsection{$\alpha$ irrational}
 In this situation to prove the completeness of $\mathbb{T}_{\alpha}$
we
adopt the notion of semi-genericity as appears in
\cite{BS97}.
\par
Recall the  notation from subsection $2.2$ that for a set $A\in \mathcal{K}_{\alpha}^{+} $  an embedding $f:A\to \mathbb{M}$ is called $\mathcal{L}_{p}$-embedding if $\qftp_{\mathcal{L}}(f(A))=\qftp_{\mathcal{L}}(A)$ and  for each $x\in A$, $x$ and  $f(x)$ have the same color.
\begin{dfn}
\label{sem-gen}
An $ \mathcal{L}_p $-structure $ M\in \mathcal{K}_{\alpha}^{+} $ is called semi-generic, if
\begin{enumerate}
\item
$ M\models T $, and
\item
whenever
$A\leqslant B$, $B$ is transcendental over
$A$ and
$f:A\to  M$ is an
$\mathcal{L}_p$-embedding  then for each natural number
$n$ there exists an $\mathcal{L}_{p}$-embedding
$\hat{f}_n: B \to  M$ extending
$f$ and such that
\begin{align*}
\cl^{n}_{M}(\hat{f}_n(B))=\hat{f}_n(B)\oplus_{f( A)} \cl^{n}_{M}(f(A)).  \quad (*)
\end{align*}
\end{enumerate}
\end{dfn}
%\begin{dfn}
%Suppose that
%$\Psi\subseteq \mathcal{K}_{\alpha}^{+}$
%is a class of finitely-generated structures,
%and $M$ is such that
%$M\subseteq N$ for all $N\in \Psi$.
%An embedding
%$f: M\to K$
%is said to
%omit
%$\Psi$ if
%it cannot be extended to any
%$g: N\to  K$
%with
%$N\in \Psi$.
%\end{dfn}
The proof of the following lemma is inspired by \cite{L07}, Proposition 4.4. .
\begin{thm}\label{lem:3.14}
Each
$\omega_{1}$-saturated model of
$\mathbb{T}_{\alpha}$ is  semi-generic.
\end{thm}
\begin{proof}
We find it beneficial to provide a
sketch of the proof beforehand.
Assuming that $A,B$ are as in item 2
in \Cref{sem-gen}, we use
the axioms and saturation to find infinitely-many mutually-free copies over $A$
of a set
$D^*$  as in \Cref{lem:3.3.4},
whose pre-dimension is close enough to that of $A$.
We will then show that there are only finitely many of these copies in which the
corresponding image of $B$ violates
$(*)$ in \Cref{sem-gen}, and therefore
there is a copy with the desired
property.
\par
Now,
assume that $\mathbb{M}$ is an
$\omega_{1}$-saturated model of $\mathbb{T}_\alpha$,
$A\leqslant  B $,
$B$ is transcendental over
$A$ and
$f:A\to  \mathbb{M}$ is an $\mathcal{L}_{p}$-embedding. Let $n$ be a  given natural number,
$\epsilon_{n}$ be as in \Cref{bound} and
 $0<\epsilon<\epsilon_{n}$. By \Cref{{lem:3.3.4}} there exists $D^{*}\in \mathcal{K}_{\alpha}^{+}$ extending
$ B$ such that
$A\leqslant D^{*}$, $D^{*}$ is transcendental over $A$ and
  $\delta(D^{*}/A)<\epsilon$. Moreover as in the mentioned lemma, $B\leqslant_{i}D^{*}$ and for every
  $B\subseteq C\subseteq D^{*}$ with  $|C\setminus B|<n$ we have
  $B\leqslant C$.
Now by applying
\Cref{obs:1.7} to $X=f(A)$, $Y=  \cl(f(A))$, and
$\Sigma(\bar{x})=\{\phi(\bar{x},f(\bar{a})):\hspace*{3mm}\phi(\bar{x},\bar{a})\in \qftp_{\mathcal{L}}(D^{*}/A)\}$ regarding the fact that
 $D^{*}$ is transcendental over $A$,
  there exists an $\mathcal{L}$-embedding $h:D^{*} \to \mathbb{M}$ such that $h(D^{*})\cap  \cl(f(A))=f(A)$ and $h(D^{*})$  and $ \cl(f(A))$ are free over $f(A)$.
\par
Now for  a finite subset $A^{'}$ of $\cl(f(A))$ including $f(A)$ we let
$D^{'}= A^{'}\oplus_{f(A)} h(D^{*})$.
We  color  the $\mathcal{L}$-structure $ \langle D^{'} \rangle_{\mathbb{M}}$ by letting $h(D^{*})$ have the same color as $D^{*}$,  while keeping the color of   $ A^{'}$  and leaving the rest non-colored.  Denote the  resulting $\mathcal{L}_{p}$-structure by $N$. Then    $D^{'}\in \mathcal{K}_{\alpha}^{+}$, since  $N\in \mathcal{K}_{\alpha}^{+}$.  Also $A^{'}\leqslant D^{'}$ and
$D^{'}$ is transcendental over $A^{'}$.
%\par
%For a natural number $n$ consider $\Delta^{n}_{A,B}$  consisting  of all $C\in \mathcal{K}_{\alpha}^{+}$ with  $B\leqslant_{i}C$, $|C-B|<n$ and there exists an $\mathcal{L}_{p}$-embedding $g: \cl_{C}(A) \to \cl(f(A))$ extending $f$.
%\par
%Let
%$\epsilon>0$ be a real number small enough ... . Then
%%  by Lemma
%%\ref{bound} there exists
%%$\epsilon>0$ such that
%% for every $E\subseteq F \in \mathcal{K}_{\alpha}^{+}$ if $|F\setminus  E|<n$ and $\delta(F/E)<0$ then
%%$\delta(F/E)<-\epsilon$. So in particular  $\delta(C/B^{'})<-\epsilon$
%%for all $ C\in \Delta^n_{A^{'},B^{'}}$.
% by Lemma
%\ref{} there is
%$D^{*}\in \mathcal{K}_{\alpha}^{+}$ extending
%$ B$ such that
%$A\leqslant D^{*}$,
%$\delta(D^{*}/B)< 0$  and  $\delta(D^{*}/A)<\epsilon$.
%\par
\par
Let $\{D'_i\}$, indexed by the natural numbers, be
disjoint copies of $D$ mutually free over $A$.
For each natural number
$m$
let
$\mathcal{D}^{'}_{m}=D^{'}_{1}\oplus_A \ldots \oplus_A D^{'}_m$ be  the free union of
$m$ disjoint
copies of
$D^{'}$
over
$A^{'}$ (so $D^{'}_{i}\cong_{A^{'}} D^{'}$). Axioms of  $\mathbb{T}_{\alpha}$  (Axiom 3) and $\omega_{1}$-saturation of $\mathbb{M}$ ensure that there exists a copy of $\mathcal{D}^{'}_{m}$
over $A^{'}$, since $\mathcal{D}^{'}_{m}$ is transcendental over $A^{'}$. Thus for each $m$ there are $m$ copies of $D^{'}$ disjoint over $f(A)$ satisfying  $\qftp_{\mathcal{L}}(h(D^{'})/A^{'})$. Hence by $\omega_{1}$-saturation there are infinity many  copies of $D^{*}$ disjoint over $f(A)$ satisfying  $\qftp_{\mathcal{L}}(h(D^{*})/\cl(f(A)))$.
 \par
Let
  $g_{i}:D^{'}\to \mathbb{M}$ be the  embedding that represents the $i$'th copy of
  $D^{'}$ over $f(A)$,
  $B_{i} =g_{i}( B)$
  and
  $ D^{*}_{i} =g_{i}( D^{'})$. By the construction,
  $D^{*}_{i}\forkindep^{\dim}_{f(A)}\cl(f(A))$ and consequently $B_{i}\forkindep^{\dim}_{f(A)}\cl(f(A))$.
\par
Let
$ Z=\{i\in \omega :\hspace*{3mm}  \cl^{n}(B_{i})\nsubseteq \cl(f(A))  \oplus_{f(A)}B_{i}\} $.
Under the assumption that
$Z$ is finite, the rest of the proof goes as follows.
 Pick some $i\not \in Z$.
  We claim  that
$\cl^{n}(B_{i})=\cl^{n}(f(A))\oplus_{f(A)} B_{i}$,
that is the condition of semi-genericity is satisfied
by $\hat{f}_n=g_i$.
 Suppose that  $E\subseteq \mathbb{M}$
 such that
$B_{i}\leqslant_{i} E$  and
$|E-B_{i}|<n$.
Then
$E\subseteq \cl(f(A))\oplus _{f(A)}B_{i}$.
If
$E_{1}=E\setminus  B_{i}$,
then
$E_{1}\subseteq \cl(f(A))$
and
$E_{1}\forkindep^{\dim}_{f(A)}B_{i}$. Therefore $f(A)\leqslant_{i} f(A)\cup E_{1}$ and $E_{1}\subseteq \cl^{n}(f(A))$. Hence
$E\subseteq \cl^{n}(f(A))\oplus _{f(A)}B_{i}$ and  $\cl^{n}(B_{i})=\cl^{n}(f(A))\oplus _{f(A)}B_{i}$.
\par
Now it remains to show that
the assumption that
  $Z$ is infinite leads to a contradiction.
   For  $i\in Z$,  denote by $C_{i}\subseteq  \mathbb{M}$ an intrinsic
  extension of
  $B_{i}$
not included in $\cl(f(A))\oplus _{f(A)}B_{i}$ with $|C_{i}\setminus B_{i}|<n$.
 Put
    $H_{i}=D^{*}_{i}\cup C_{i}$ and choose a natural number
    $s$  such that
    $|H_{i}|<s$ for all
    $i\in Z$. As every subset of $D^{*}_{i}$ that includes $B_{i}$ and has  less than $n$ more elements than $B_{i}$ is closed in $D^{*}_{i}$, we have  $C_{i}\nsubseteq D^{*}_{i}$.
%\par
%We now show that
%on the one hand,
%by a direct calculation
%$\delta(H_i/A')<0$, and on the other hand,
%as a consequence of
%\ref{lem:1.22} we have $\delta(H_i/A')\geqslant 0$, and this is a contradiction.
    \par
 By Lemma
    \ref{lem:1.22} there are an arbitrary large finite set
    $\Omega\subseteq Z$
and a set
   $F$ such that
   $H_{i}\cap H_{j}=F$
   and
   $F\leqslant H_{i}$  for
  all $ i\neq j\in \Omega$.     If $A^{'}=F\cap \cl(f(A))$ then $A^{'}\leqslant F$ and
%   $i\in \Omega$  we have
%   $f(A)\subseteq F\subseteq h_{i}(C_{i})$,
%   and the latter is because of the  disjointness of   the
%   $D_{i}^{*}$ over
%   $f(A)$. But since
%   $f(A)\leqslant C_{i}$, it follows that
%   $f(A)\leqslant F$. Therefore
   by transitivity
   $A^{'}\leqslant H_{i}$ and
   $\delta(H_{i}/A^{'})\geqslant 0$.
   \par
On the other hand  we have that
   $ C_{i}\setminus D^{*}_{i}\neq \emptyset$. So
  \begin{align*}
  \delta(H_{i}/A^{'})=\delta(D^{*}_{i}\cup C_{i}/A^{'})&=\delta(D^{*}_{i}/A^{'})+\delta(C_{i}/D^{*}_{i}).
  \end{align*}
  Notice that
  $\delta(C_{i}/D^{*}_{i})<0$. Hence as we have $|C_{i}\setminus D^{*}_{i}|<n$, it follows that $\delta(C_{i}/D^{*}_{i})<-\epsilon_{n}<-\epsilon$.
  Furthermore as $D^{*}_{i}\forkindep^{\dim}_{f(A)}\cl(f(A))$, we have that
  $\delta(D^{*}_{i}/A^{'})=\delta(D^{*}_{i}/A)<\epsilon$.
 Hence
 $\delta(H_{i}/A^{'})<0$, a contradiction.
\end{proof}
\begin{dfn}
By $X\stackrel{w}{\cong} Y$,
read as $X$ is weakly isomorphic to $Y$,
we mean that there is an
$\mathcal{L}$-isomorphism $f:\langle X\rangle\to \langle Y\rangle$
mapping $X$ to $Y$ and such that $p(x)\leftrightarrow p(f(x))$ for
all $x\in X$.
\end{dfn}
\begin{rem}\label{wiso}
Notice that the above notion of
isomorphism is weaker than
$\langle X \rangle\cong _{\mathcal{L}_p} \langle Y\rangle$, since the colours are preserved only for elements of $X$.
Weak isomorphism of $X$ and $Y$, as in the definition above, also
implies that for every $A\subseteq X$ and natural number $i$, if $\cl^i_{\langle X\rangle}(A)\subseteq X$ and  $\cl^i_{\langle Y\rangle}(f(A))\subseteq Y$ then $f(\cl^i_{\langle X\rangle}(A))=\cl^i_{\langle Y\rangle}(f(A))$ and therefore $\cl^i_{\langle X\rangle}(A)\stackrel{w}{\cong} \cl^i_{\langle Y\rangle}(f(A))$.
\end{rem}
\begin{thm}\label{lem:3.15}
Let
$\mathbb{M}_{1}$ and $\mathbb{M}_{2}$ be two $\omega_{1}$-saturated  models of
$\mathbb{T}_{\alpha}$
and $\phi(\bar{x})$ an $\mathcal{L}_p$-formula.
Then
there exists
$n=n_{\varphi}$ such that for
all $\bar{a}_1\in \mathbb{M}_1^{|\bar{a}_{1}|}$
and $\bar{a}_2\in \mathbb{M}_2^{|\bar{a}_{2}|}$,
\begin{align*}
 \cl^n_{\mathbb{M}_1}(\bar{a}_1) \stackrel{w}{\cong}\cl^n_{\mathbb{M}_2}(\bar{a}_2)\Rightarrow
\big (\mathbb{M}_{1}\models \varphi(\bar{a}_{1}) \Leftrightarrow \mathbb{M}_{2}\models \varphi(\bar{a}_{2})\big).
\end{align*}
\end{thm}
Note that the idea of the following proof is  borrowed from
\cite{BS97}, Lemma 1.30, but there are essential alterations
that make the  rewrite necessary.
\begin{proof}
Assume that
$\mathbb{M}_{1},\mathbb{M}_{2}$,
$\bar{a}_{1}\in \mathbb{M}_{1}^{|\bar{x}|}$ and
$\bar{a}_{2}\in \mathbb{M}_{2}^{|\bar{x}|}$ are as in the statement of the theorem.
We proceed by induction  on the  complexity of $\mathcal{L}_{p}$-formulas. As it appears, the induction step for the
$\mathcal{L}$-atomic formulas
as well as the
boolean connectives are straightforward.  For an $\mathcal{L}_{p}$-atomic formula of the form $\varphi:=p(t(x_{1},\dots, x_{n}))$ where $t(x_{1},\dots, x_{n})$ is an $\mathcal{L}$-term, we may take $n_{\varphi}=1$. Then notice that $\acl_{\mathbb{M}_{i}}(\bar{a}_{i})\cap p(\mathbb{M}_{i})\subseteq \cl^{1}_{\mathbb{M}_{i}}(\bar{a}_{i})$ for $i=1,2$. Hence   $\cl^{1}_{\mathbb{M}_{1}}(\bar{a}_{1})\stackrel{w}{\cong} \cl^{1}_{\mathbb{M}_{2}}(\bar{a}_{2})$ implies that $\acl_{\mathbb{M}_{1}}(\bar{a}_{1}) \stackrel{w}{\cong} \acl_{\mathbb{M}_{2}}(\bar{a}_{2})$. Therefore $t^{\mathbb{M}_{1}}(\bar{a}_{1})\in p(\mathbb{M}_{1})$ if and only if  $t^{\mathbb{M}_{2}}(\bar{a}_{2})\in p(\mathbb{M}_{2})$.
\par
 Now assuming  that  the statement holds for
$\psi(y, \bar{x})$ we need to introduce a natural number
$n_{\varphi}$ so that
it holds also for the  formula
$\varphi(\bar{x})=\exists y \psi(y,\bar{x})$.
\par
By Remark \ref{rem:2.17}
 let the natural number $p_{1}$ be   sufficiently large so that
$|\cl^{n_{\psi}}_{N}(\bar{a}d)|<p_{1}$
for all
$N\in \mathcal{K}_{\alpha}^{+}$ and $\bar{a}, d\in N$ with $|\bar{a}|=|\bar{x}|$.
 Put $p=\max\{p_{1},n_{\psi}\}$.
For $N\in \mathcal{K}_{\alpha}^{+}$
and $\bar{a}\in N^{|\bar{x}|}$ and $i< p$,  we define inductively
 $A_i^N(\bar{a})$ by setting $A_{0}^{N}(\bar{a})=\bar{a}$ and $ A^{N}_{i+1}(\bar{a})=\cl^{p}_{N}(A_{i}^N(\bar{a}))$. Now let $n_{\varphi}\geqslant 1$ be such that
$A^{N}_{p}(\bar{a})\subseteq \cl^{n_{\varphi}}_{N}(\bar{a})$
 for all $N\in \mathcal{K}^{+}_{\alpha}$ and
$\bar{a}\in N^{|\bar{x}|}$. We will show that $n_{\varphi}$ satisfies the requirement of the theorem.
\par
 Let
$\bar{a}_{1}\in \mathbb{M}_{1}^{|\bar{x}|}$ and
$\bar{a}_{2}\in \mathbb{M}_{2}^{|\bar{x}|}$ be such that
$ \cl^{n_{\varphi}}_{\mathbb{M}_{1}}(\bar{a}_{1}) \stackrel{w}{\cong}  \cl^{n_{\varphi}}_{\mathbb{M}_{2}}(\bar{a}_{2}) $.
For
$b_{1}\in \mathbb{M}_{1}$ we show that
there is
$b_{2}\in \mathbb{M}_{2}$ with
$\cl^{n_{\psi}}_{\mathbb{M}_{1}}(\bar{a}_{1}b_{1})\stackrel{w}\cong \cl^{n_{\psi}}_{\mathbb{M}_{2}}(\bar{a}_{2}b_{2})$.
\par
Denote by $H_0$ the set with elements
$\bar{a}_1b_1$ and let
$H_{1}=\cl^{n_{\psi}}_{\mathbb{M}_{1}}(H_{0})$.
Since $ \cl^{n_{\varphi}}_{\mathbb{M}_{1}}(\bar{a}_{1}) \stackrel{w}{\cong}  \cl^{n_{\varphi}}_{\mathbb{M}_{2}}(\bar{a}_{2}) $,
for each
$i\leqslant p$,
$A_i^{\mathbb{M}_{1}}:=A_{i}^{\mathbb{M}_{1}}(\bar{a}_{1})\stackrel{w}{\cong } A_{i}^{\mathbb{M}_{2}}:=A_{i}^{\mathbb{M}_{2}}(\bar{a}_{2})$. Since $|H_{1}\setminus H_{0}|<p$, there is
$j< p$ such that
\begin{align*}
(A^{\mathbb{M}_{1}}_{j+1}\setminus A^{\mathbb{M}_{1}}_{j})\cap (H_{1}\setminus H_{0})=\emptyset.
\end{align*}
Moreover the fact that
$p > |H_{1}\setminus A^{\mathbb{M}_{1}}_{j}|$
  implies
 $A^{\mathbb{M}_{1}}_{j}\leqslant  A^{\mathbb{M}_{1}}_{j} H_{1}$.
\par
Now  by Remark  \ref{rem:2.19} we may find $G_{1}\subseteq H_{1}$ such that $A^{\mathbb{M}_{1}}_{j}G_1$ is algebraic over $A^{\mathbb{M}_{1}}_{j}$ and  $A^{\mathbb{M}_{1}}_{j} H_{1}$ is transcendental over $A^{\mathbb{M}_{1}}_{j}G_1$.
Note that since
$n_{\varphi}\geqslant 1$, we have
$\acl_{\mathbb{M}_1}(\bar{a}_1)\stackrel{w}{\cong }\acl_{\mathbb{M}_2}(\bar{a}_2)$.
Hence there is $G_2\subseteq \mathbb{M}_2$ such that $A_j^{\mathbb{M}_1}G_1 \stackrel{w}{\cong} A_j^{\mathbb{M}_2}G_2$.
Furthermore, again by Remark \ref{rem:2.19},
\begin{align*}\cl^p_{\mathbb{M}_2}(A_j^{\mathbb{M}_2}G_2)=\cl^{p}_{\mathbb{M}_2}(A_j^{\mathbb{M}_2})\oplus_{A_j^{\mathbb{M}_2}} A_j^{\mathbb{M}_2} G_2.\hspace*{3mm}(\ast)
\end{align*}
Let
$g:A^{\mathbb{M}_{1}}_{j} G_{1}  \to \mathbb{M}_{2}$ be the $\mathcal{L}_{p}$-embedding witnessing  $A_j^{\mathbb{M}_1}G_1 \stackrel{w}{\cong} A_j^{\mathbb{M}_2}G_2 $.
Now since   $ A^{\mathbb{M}_{1}}_{j}  H_{1}$ is transcendental  over $A^{\mathbb{M}_{1}}_{j}G_{1}$, by  semi-genericity of $\mathbb{M}_2$, we find an $\mathcal{L}_p$-embedding
$\hat{g}:A_j^{\mathbb{M}_1} H_{1}\to \mathbb{M}_{2}$ extending $g$ such that
\begin{align*}
  \cl^p_{\mathbb{M}_2}(A_j^{\mathbb{M}_2}\hat{g}(H_{1}))=\cl^{p}_{\mathbb{M}_2}(A_j^{\mathbb{M}_2}G_2)\oplus_{A_j^{\mathbb{M}_2}G_2} A_j^{\mathbb{M}_2}\hat{g}(H_1).
\end{align*}
 Now by $(\ast)$ we have
\begin{align*}
\cl^p_{\mathbb{M}_2}(A_j^{\mathbb{M}_2}\hat{g}(H_1))=\cl^{p}_{\mathbb{M}_2}(A_j^{\mathbb{M}_2})\oplus_{A_j^{\mathbb{M}_2}} A_j^{\mathbb{M}_2} \hat{g}(H_1),
\end{align*}
that is
\[
\cl^p_{\mathbb{M}_2}(A_j^{\mathbb{M}_2}\hat{g}(H_1))=A_{j+1}^{\mathbb{M}_2}\oplus_{A_j^{\mathbb{M}_2}} A_j^{\mathbb{M}_2} \hat{g}(H_1).
\]
Now
letting
$b_{2}=\hat{g}(b_{1})$,
we will show that
$\cl^{n_{\psi}}_{\mathbb{M}_{1}}(\bar{a}_{1}b_{1})\stackrel{w}\cong \cl^{n_{\psi}}_{\mathbb{M}_{2}}(\bar{a}_{2}b_{2})$.
\par
By the definition of $\hat{g}$  we have
$A_{j}^{\mathbb{M}_{2}}\hat{g}(H_{1})  \stackrel{w}{\cong}A_{j}^{\mathbb{M}_{1}}H_{1}$  and  $\cl_{\mathbb{M}_{1}}^{n_{\psi}}(\bar{a}_{1},b_{1})=H_1\subseteq A_{j}^{\mathbb{M}_{1}}H_{1}$. Note that
\begin{align*}
 \cl_{\mathbb{M}_{2}}^{n_{\psi}}(\bar{a}_{2},b_{2})\subseteq\cl_{\mathbb{M}_{2}}^{p}(\bar{a}_{2},b_{2})\subseteq \cl_{\mathbb{M}_{2}}^{p}(A_{j}^{\mathbb{M}_{2}}\hat{g}(H_{1}))=A_{j+1}^{\mathbb{M}_{2}}\hat{g}(H_{1}).
\end{align*}
 Moreover, since
$A_{j+1}^{\mathbb{M}_{2}}$ and $\hat{g}(H_{1})$ are free over $A_{j}^{\mathbb{M}_{2}}$,
we have
$ \cl_{\mathbb{M}_{2}}^{n_{\psi}}(\bar{a}_{2},b_{2})\subseteq A_{j}^{\mathbb{M}_{2}}\hat{g}(H_1)$.
%by we see that
%$\cl^{n_{\psi}}_{A^{\mathbb{M}_{2}}_{j+1} \hat{g}(H_{1})} (\bar{a}_{2},b_{2}) \subseteq A^{\mathbb{M}_{2}}_{j}\hat{g}(H_{1})$ hence
So it follows that
 $\cl^{n_{\psi}}_{\mathbb{M}_{1}} (\bar{a}_{1},b_{1})=\cl^{n_{\psi}}_{\langle A^{\mathbb{M}_{1}}_{j}H_{1}\rangle} (\bar{a}_{1},b_{1})$
 and
$\cl^{n_{\psi}}_{\mathbb{M}_{2}} (\bar{a}_{2},b_{2})=\cl^{n_{\psi}}_{\langle A^{\mathbb{M}_{2}}_{j} \hat{g}(H_{1})\rangle} (\bar{a}_{2},b_{2})$. On the other hand  by \Cref{wiso} we have that
\begin{align*}
\cl^{n_{\psi}}_{\langle A^{\mathbb{M}_{1}}_{j}H_{1}\rangle} (\bar{a}_{1},b_{1})\stackrel{w}{\cong} \cl^{n_{\psi}}_{\langle A^{\mathbb{M}_{2}}_{j} \hat{g}(H_{1})\rangle} (\bar{a}_{2},b_{2}).
\end{align*}
Therefore
$\cl^{n_{\psi}}_{\mathbb{M}_{1}} (\bar{a}_{1},b_{1}) \stackrel{w}{\cong}\cl^{n_{\psi}}_{\mathbb{M}_{2}} (\bar{a}_{2},b_{2})$.
\end{proof}
The completeness of $\mathbb{T}_\alpha$
is finally established as follows.
\begin{thm}
$\mathbb{T}_{\alpha}$ is complete.
\end{thm}
\begin{proof}
If $\mathbb{M}_{1}$  and
$\mathbb{M}_{2}$ are two
$\omega_{1}$-saturated models of
$\mathbb{T}_{\alpha}$ then by Theorem \ref{lem:3.15},
$\mathbb{M}_{1}\equiv \mathbb{M}_{2}$. Therefore $\mathbb{T}_{\alpha}$ is complete.
\end{proof}
\begin{cor}
\label{types-closures}
 Let
$\mathfrak{C}$
be a monster model of
$\mathbb{T}_{\alpha}$.
\begin{enumerate}
\item
Assume that
$\bar{a}_{1},\bar{a}_{2}$ in $\mathfrak{C}$ are small tuples
 and
 $X$ is a  closed small subset of $\mathfrak{C}$. Then,
\begin{align*}
\tp(\bar{a}_{1}/X)=\tp(\bar{a}_{2}/X)\Leftrightarrow\langle \cl(X\bar{a}_{1})\rangle \cong_{\langle X \rangle} \langle \cl(X\bar{a}_{2})\rangle
\end{align*}
\item
$\mathfrak{C}$ is
$\lambda$-rich, for all
$\lambda<|\mathfrak{C}|$.
\end{enumerate}
\end{cor}
\subsubsection{$\alpha$ rational}
When
$\alpha$ is rational, we show the completeness of $\mathbb{T}_\alpha$
 by proving that
any
 $\omega$-saturated model of $\mathbb{T}_{\alpha}$ is
$\omega$-rich and that any two $\omega$-rich structures of $\mathcal{K}^{+}_{\alpha}$ are back and forth equivalent.
\begin{thm}\label{th:3.4.6}
 Suppose that $\mathbb{M}$ is an $\omega$-saturated model of $\mathbb{T}_{\alpha}$ then $\mathbb{M}$ is $\omega$-rich.
\end{thm}
\begin{proof}
Let
$M\leqslant N$ be two finitely generated structures in $\mathcal{K}_{\alpha}^{+}$
and
$f:M\to \mathbb{M}$ be a strong $\mathcal{L}_{p}$-embedding. We prove that there is a strong $\mathcal{L}_{p}$-embedding $g:N\to \mathbb{M}$  extending
$f$.
\par
Since $M,N$ are finitely generated, there is $A\subseteq_{\text{fin}}M$ and $B\subseteq_{\text{fin}}N$ such that
$M=\langle A\rangle$, $N=\langle B\rangle$ and
$A\leqslant B\leqslant N$. By \Cref{rem:2.19} we have two  specific cases to consider.
\paragraph*{Case 1.}
$B$ is   algebraic over $A$.
\par \noindent
Since $A\leqslant N\in \mathcal{K}_{\alpha}^{+}$, any  $x\in B\setminus A$ is non-colored. So by \Cref{rem:2.19}  it is clear that there is a strong $\mathcal{L}_{p}$-embedding  $g:B\to \mathbb{M}$ extending $f$.
\paragraph*{Case 2.}
$B$ is   transcendental over  $A$.
\par \noindent
There are two subcases  to consider:
\par
\textbf{Case 2.1}
$\delta(B/A)>0$. By \Cref{lem:3.3.6} there is   $D\in \mathcal{K}_{\alpha}^{+}$ extending $B$  with $A\leqslant D$,  $D$  transcendental over $A$  and $\delta(D/A)=0$. Using the same method as in the  proof of \Cref{lem:3.14}   for each natural number $n$ we can find an $\mathcal{L}_{p}$-embedding $g:D\to \mathbb{M}$ extending $f$   such that $g(B)$ is $n$-strong, i.e. for each $C\subseteq \mathbb{M}$  containing $g(B)$ with $|C\setminus g(B)|<n$  we have $B\leqslant C$. Hence by $\omega$-saturation there is indeed a strong $\mathcal{L}_{p}$-embedding $g:B\to \mathbb{M}$ extending $f$.
\par
\textbf{Case 2.2.}
$\delta(B/A)=0$. This can be dealt with as Case $2.1$ by  taking $D$ above in the place of $B$.
\par
Now  by \Cref{rem:2.19} there is $B_{1}$ with $A\leqslant B_{1}\leqslant B$ such that
$B_{1}$ is algebraic over $A$ and $B$ is transcendental over $B_{1}$. By  Case $1$ there is an $\mathcal{L}_{p}$-embedding $g_{1}:B_{1}\to \mathbb{M}$ extending $f$   and by Case 2, there is an  $\mathcal{L}_{p}$-embedding
  $g:B\to \mathbb{M}$ extending $g_{1}$.
\end{proof}
It is easy to see that $\omega$-rich structures are back
and forth equivalent and hence so are any two
$\omega$-saturated models of $\mathbb{T}_\alpha$.
This yields the following sought-for theorem.
\begin{thm}
$\mathbb{T}_{\alpha}$ is complete.
\end{thm}

\begin{rem}
\Cref{types-closures} also holds for rational $\alpha$.
\end{rem}
 \begin{rem}
 If $T$ is an o-minimal expansion of a densely ordered group and $\alpha=1$ then  $\mathbb{T}_{\alpha}$ happens to be the theory $T^{\indep}$ as introduced in \cite{D16}.
 \end{rem}
\section{Dependence and Strong Dependence}
In this section, we will prove that
if $T$ is
dependent, that is
when it has Not-the-Independence Property,
 then so
is
$\mathbb{T}_{\alpha}$.
We will show that
 $\mathbb{T}_{\alpha}$ inherits the strong dependence from $T$ as well, when
 $\alpha$ is rational. Moreover
we will show that
when $T$ defines a linear order,
 $\mathbb{T}_\alpha$ is not strongly dependent.
\par
Recall that
 an
$\mathcal{L'}$-theory
$T'$ is called
\textit{dependent} if each
$\mathcal{L'}$-formulas
$\varphi(\bar{x},\bar{y})$ is dependent,
that is there are no model
$M'\models T'$,
and sequences
$\{\bar{a}_{i}:\hspace*{3mm}i\in \omega\}$
and
$\{\bar{b}_{J}:\hspace*{3mm}J\subseteq \omega\}$
such that
$M\models \varphi(\bar{a}_{i},\bar{b}_{J})$
if and only if
$i\in J$.
It is known that
in this definition suffices it to assume
$|\bar{x}|=1$.
\par
From now on we  add the further assumption that
$T$ is  dependent, and let
$\mathfrak{C}$
be a monster model of
$\mathbb{T}_{\alpha}$. In this section we use
boldface letters $\mathbb{M},\mathbb{N}$
for models of $\mathbb{T}_\alpha$.
\par
By standard facts on dependent theories,
to show that
$\mathbb{T}_{\alpha}$ is dependent, we
verify that for a given model
$\mathbb{M}$ of
$\mathbb{T}_{\alpha}$, the number of coheir extensions of any $1$-type
$p(x)$ over
$\mathbb{M}$ is  at most
$2^{|\mathbb{M}|}$.
The following notion of
$D$-independence
would be of great help (\cite{P15}) for our counting of the coheirs.
\begin{dfn}
Let
$\mathbb{M}$
be an arbitrary model of
$\mathbb{T}_{\alpha}$,
$A,B\subseteq_{\text{fin}} \mathbb{M}$
and $X\subseteq \mathbb{M}$. Define
\begin{enumerate}
\item $D(A)=\inf\{\delta(C) :\hspace*{3mm}A\subseteq C\subseteq_{\text{fin}}\mathbb{M}\}$,
\item
$D(B/A)=D(BA)-D(A)$,
\item
$\CL(A)=\{x\in \mathbb{M}:\hspace*{3mm} D(xA) = D(A)\}$,
\item
$D(A/X)=\inf\{D(A/X_{0}):\hspace*{3mm} X_{0}\subseteq_{\text{fin}}X\}$.
\end{enumerate}
\end{dfn}
Note that the function $D$ is automorphism invariant, i.e. for any automorphism $f:\mathfrak{C}\to \mathfrak{C}$ we have $D(A/X)=D(f(A)/f(X))$  for any $A,X\subseteq  \mathfrak{C}$.
%$\CL(A)$ is called the $D$-closure of $A$.
%It is clear that
%$\cl(A)\subseteq \CL(A)$.
%It is known that the $D$-closure gives rise to
%a pre-geometry on
%$\mathbb{M}$ whose
%corresponding notion of dimension is $D$.
\begin{fact}
\label{no-decreasing-sequence-of-D}
When $\alpha$ is rational it can be easily seen that
$D(A)=\delta(\cl(A))$
and $D(B/A)=
\delta(\cl(BA))-\delta(\cl(A))$. Therefore
the set containing all positive values
  $D(B/A)$ is discrete, and it does not contain any infinite decreasing  sequence.
\end{fact}
\begin{dfn}
Let
$\mathbb{M}\models \mathbb{T}_{\alpha}$,
$A, B\subseteq_{\text{fin}}\mathbb{M}$
 and
$Z\subseteq \mathbb{M}$. We say that
$A, B$
are $D$-independent over
$Z$
and write
$A\forkindep_{Z}^{D} B$ whenever
$D(A/Z)=D(A/ZB)$
and
$\cl(AZ)\cap \cl(BZ)=\cl(Z)$.
Moreover for arbitrary subsets
$X, Y$
of
$\mathbb{M}$
we say that
$X, Y$
are $D$-independent over
$Z$
and write
$X\forkindep_{Z}^{D}Y$
if
$C\forkindep_{Z}^{D} E$
for any
$C\subseteq_{\text{fin}} X$
and
$E\subseteq_{\text{fin}} Y$.
In this case we also say that
$\tp(X/YZ)$ is a $D$-independent extension of $\tp(X/Z)$.
\end{dfn}
The following facts,
and the technique to prove the following lemma,
are available in
\cite{{BS96}}, Section  3.
\begin{fact}
\label{fact-4-3}
The relation $\forkindep^{D}$ has the following properties.
\begin{enumerate}
\item\textit {$D$-symmetry}.
If
$A\forkindep_{B}^{D} C$, then
$C\forkindep_{B}^{D} A$.
\item\textit {$D$-transitivity}.
$A\forkindep_{B}^{D} C$
and
$A\forkindep_{BC}^{D} E$ if and only if
$A\forkindep_{B}^{D} CE$.
\item\textit {$D$-local character}.
For any $A$ and
$X$
there exists a countable set
$X_{0}\subseteq X$ such that
$A\forkindep_{X_{0}}^{D} X$.
In the case of a
rational $\alpha$,  $X_0$
is finite.
\item\textit {$D$-existence}.
For all sets
$X, Y, Z $
with
$Y\subseteq  Z$
there exists
$X^{'}$ such that
$\tp(X/Y)=\tp(X^{'}/Y)$
 and
 $X^{'}\forkindep_{Y}^{D} Z$.
\end{enumerate}
\end{fact}
\begin{fact}
\label{morley-sequence}
By $D$-existence,
it is possible to find a
$D$-Morley sequence for
a given type
$p(\bar{x})\in S_{|\bar{x}|}(M)$; that is  a sequence
$(\bar{b}_{i})_{i<\omega}$ of realizations of
$p(\bar{x})$ such that
$\bar{b}_{0},\ldots,\bar{b}_{n}\forkindep_{M}^{D} \bar{b}_{n+1}$ for each
$n<\omega$.
\end{fact}
\begin{lem}\label{lem:4.4}
Let $Z\leqslant  \mathbb{M}$
and
$X,Y\subseteq \mathbb{M}$.
Then $
X\forkindep_{Z}^{D} Y$
if and only if
\begin{enumerate}
\item
$\cl(XZ)\cap \cl(YZ)=Z$,
\item
$\cl(XY)=\cl(XZ)\cup \cl(YZ)$,
\item
$\cl(XZ)\forkindep_{Z}^{\dim}\cl(YZ)$.
\end{enumerate}
The above three conditions imply that
$\cl(XY)=\cl(XZ)\oplus_{Z} \cl(YZ)$.
\end{lem}
\begin{lem}\label{lem:4.5}
Suppose that
$q(\bar{x})\in S(\mathfrak{C})$ is  a coheir extension of
$p(\bar{x})\in S(\mathbb{M})$. Then
$q(\bar{x})$
is a
$D$-independent extension of
$p(\bar{x})$.
\end{lem}
\begin{proof}
For
$\bar{a}\models q(\bar{x})$, we need to show that
$\bar{a}\forkindep_{\mathbb{M}}^{D}\mathfrak{C} $. Take a small set
$\mathbb{N}$ with
$\mathbb{M}\leqslant \mathbb{N}\leqslant \mathfrak{C}$  and
$D(\bar{a}/\mathbb{N})=D(\bar{a}/\mathfrak{C})$.
By $D$-existence and $D$-transitivity   (Fact \ref{fact-4-3}) we can  also find
$\mathbb{N}^{'}$ such that
$\mathbb{N}^{'}\forkindep_{\mathbb{M}}^{D}\mathbb{N}$ and
$\mathbb{N}^{'}\equiv _{\mathbb{M}}\mathbb{N}$.
\paragraph{Claim.}
$D(\bar{a}/\mathbb{N}^{'})=D(\bar{a}/\mathbb{N})$.
\newline
\textit{Proof of claim.}
Otherwise, assume that
$D(\bar{a}/\mathbb{N}^{'})=\gamma>D(\bar{a}/\mathbb{N})$. So we can find
$\bar{b} \subseteq_{\text{fin}} \cl(\bar{a}\mathbb{N})$ and
$\bar{e}\subseteq_{\text{fin}}\mathbb{N}$
such that
$\delta(\bar{b}/\bar{e})<\gamma$.
By Fact \ref{fact:1.1}
 there is
$\varphi( \bar{z},\bar{x},\bar{e})\in \tp(\bar{b}\bar{a}/\bar{e})$
 such that
for any
$\bar{b}^{'}$
and
$\bar{e}^{'}$
if
$\mathfrak{C}\models \varphi(\bar{b}^{'},\bar{a}, \bar{e}^{'})$
 then
$\bar{b}^{'}\subseteq \cl(\bar{a}\bar{e}^{'})$
and
$\delta(\bar{b}^{'}/\bar{a}\bar{e}^{'})<\gamma$.
So the formula
$\exists\bar{z} \hspace*{2mm}\varphi( \bar{z},\bar{x},\bar{e})$ is in  $\tp(\bar{a}/\mathfrak{C}) $. But since
$\tp_{\mathcal{L}}(\bar{a}/\mathfrak{C})$ is a coheir extension of
$\tp_{\mathcal{L}}(\bar{a}/\mathbb{M})$,
it is also
$\mathbb{M}$-invariant. So there is
$\bar{e}^{'}\subseteq \mathbb{N}^{'} $ such that
$\exists\bar{z} \hspace*{2mm}\varphi( \bar{z},\bar{x},\bar{e}^{'})\in \tp_{\mathcal{L}}(\bar{a}/\mathfrak{C})$. Therefore
$\delta(\bar{b}^{'}/\bar{a}\bar{e}^{'})<\gamma$ for some
$\bar{b}^{'}\subseteq \cl(\bar{a}\mathbb{N})$. This implies that
$D(\bar{a}/\mathbb{N}^{'})<\gamma$, a contradiction.\\
\qed(Proof of claim)\\
Writing the above claim as
$
\bar{a}\forkindep^D_{\mathbb{N}'} \mathbb{N}
$
together with our assumption that
$
\mathbb{N}'\forkindep^D_{\mathbb{M}}\mathbb{N}
$
and $D$-transitivity gives $\mathbb{N}'\bar{a}\forkindep^D_{\mathbb{M}}\mathbb{N}$.
Therefore again by $D$-transitivity we have  $\bar{a}\forkindep^D_{\mathbb{M}}
\mathfrak{\mathbb{N}}$, and hence
$\bar{a}\forkindep^D_{\mathbb{M}}\mathfrak{C}$
because
$D(\bar{a}/\mathfrak{C})=D(\bar{a}/\mathbb{N})$.

Finally we have to show that
$\cl(\bar{a}\mathbb{M})\cap \mathfrak{C}=\mathbb{M}$. Let
$e\in \cl(\bar{a}\mathbb{M})\cap \mathfrak{C}$.  By Remark \ref{rem:2.17}  there exists an algebraic  $\mathcal{L}_{p}$-formula
$\varphi(\bar{x}, e,\bar{m})\in q(\bar{x})$  witnessing  $e\in \cl(\bar{a}\mathbb{M})\cap \mathfrak{C}$. If $\tp(e/\mathbb{M})$ has  infinitely many realizations $(e_{i})_{i<\omega}$ in $\mathfrak{C}$ then as $q(\bar{x})$ is a coheir extension of $p(\bar{x})$, the formula
$\varphi(\bar{a},e_{i},\bar{m})$ holds for each $i<\omega$,
which contradicts the fact that $e\in  \acl(\bar{a}\mathbb{M})$.
 Therefore $\tp(e/\mathbb{M})$ is algebraic and hence
$e\in \mathbb{M}$.
\end{proof}
Recall that coheir extensions are in particular non-forking.
In the following proposition using a similar idea as in \cite{G05} Lemma 4.3,  we address the fact that  when $\alpha$ is rational, the above lemma holds even for
non-dividing extensions. However we would not require
this in the rest.
\begin{prop}
($\alpha$ rational)
Assume that
$\mathbb{M},\mathbb{N}$
are models of
$\mathbb{T}_{\alpha}$
such that
$\mathbb{M}\leqslant \mathbb{N}\leqslant \mathfrak{C}$. If
$\tp(\bar{a}/\mathbb{N})$
is a non-dividing extension of
$\tp(\bar{a}/\mathbb{M})$,
then
$\bar{a}\forkindep_{\mathbb{M}}^{D}\mathbb{N}$.
\end{prop}
\begin{proof}
We first show that
$D(\bar{a}/\mathbb{M})=D(\bar{a}/\mathbb{N})$.
Assume on the contrary that
$D(\bar{a}/\mathbb{M})>D(\bar{a}/\mathbb{N})$. So  there is
$\bar{b}\subseteq \mathbb{N}\setminus \mathbb{M}$ such that
$D(\bar{a}/\mathbb{M})>D(\bar{a}/\mathbb{M}\bar{b})$. Assume that
$(\bar{b}_{i})_{i<\omega}$ is a Morley sequence of
$\tp(\bar{b}/\mathbb{M})$ with respect to
$\forkindep_{\mathbb{M}}^{D}$, as in
Fact \ref{morley-sequence}.
Clearly
$D(\bar{a}/\mathbb{M})>D(\bar{a}/\mathbb{M}\bar{b_{i}})$
for each
$i<\omega$.
 Now by the
$D$-symmetry  $D(\bar{b}_{i}/\mathbb{M})>D(\bar{b}_{i}/\mathbb{M}\bar{a})$.  On the other hand,
\begin{align}
D(\bar{b}_{i}/\mathbb{M}\bar{b}_{0},\ldots, \bar{b}_{i-1})\geqslant D(\bar{b}_{i}/\mathbb{M}\bar{a}\bar{b}_{0},\ldots, \bar{b}_{i-1}).
\end{align}
Now if equality occurred in  (1)  then we would have
\begin{align*}
D(\bar{b}_{i}/\mathbb{M})&=D(\bar{b}_{i}/\mathbb{M}\bar{b}_{0},\ldots, \bar{b}_{i-1})\\ &=D(\bar{b}_{i}/\mathbb{M}\bar{a}\bar{b}_{0},\ldots, \bar{b}_{i-1})\\
&\leqslant D(\bar{b}_{i}/\mathbb{M}\bar{a})
\end{align*}
which  is  a contradiction. So we have
\begin{align*}
D(\bar{b}_{i}/\mathbb{M}\bar{b}_{0},\ldots, \bar{b}_{i-1})> D(\bar{b}_{i}/\mathbb{M}\bar{a}\bar{b}_{0},\ldots, \bar{b}_{i-1}).
\end{align*}
Writing the above as
$\bar{b}_{i} \not\forkindep_{\mathbb{M}\bar{b}_{0},\ldots, \bar{b}_{i-1}}^{D}\bar{a}$, again by  the $D$-symmetry
we have
$\bar{a}\not\forkindep_{\mathbb{M}\bar{b}_{0},\ldots, \bar{b}_{i-1}}^{D} \bar{b}_{i}$. That is
\begin{align*}
D(\bar{a}/\mathbb{M}\bar{b}_{0},\ldots,\bar{b}_{i-1})>D(\bar{a}/\mathbb{M}\bar{b}_{0},\ldots,\bar{b}_{i}),
\end{align*}
 for each $i<\omega$, which is a contradiction,
 because in the case of a rational $\alpha$
 we cannot have a decreasing sequence of $D$-dimensions
 (see Fact \ref{no-decreasing-sequence-of-D}).
 \par
 Now the fact that
 $\cl(\bar{a}\mathbb{M})\cap \mathbb{N}=\mathbb{M}$
follows from a similar argument to the proof of Lemma \ref{lem:4.5}.
\end{proof}
We now get to the point to prove that the dependence is transferred  from $T$ to $\mathbb{T}_{\alpha}$.
\begin{theo}\label{th:4.7}
If
$T$
is
dependent,
then
so is
$\mathbb{T}_{\alpha}$.
\end{theo}
\begin{proof}
Let
$\mathbb{M}$ and
$\mathbb{N}$ be models of
$\mathbb{T}_{\alpha}$ such that
$\mathbb{M}\leqslant \mathbb{N}\leqslant \mathfrak{C}$. Assume that
$q_{1}(x), q_{2}(x)\in S_{1}(\mathbb{N})$ are two coheir extensions of a type
$p(x)\in S_{1}(\mathbb{M})$.
If
$a_{1}\models q_{1}(x)$
and
$a_{2}\models  q_{2}(x)$
then by Corollary \ref{types-closures}
there is an $\mathcal{L}_{p}$-isomorphism between
$\cl(a_{1}\mathbb{M})$ and $\cl(a_{2}\mathbb{M})$ mapping
$a_{1}$ to $a_{2}$ and fixing $\mathbb{M}$ pointwise, and
therefore
$\tp(\cl(a_{1}\mathbb{M})/\mathbb{M})=\tp(\cl(a_{2}\mathbb{M})/\mathbb{M})$.
Since $\cl(a_{i}\mathbb{M})$
is contained in the $\mathcal{L}_{p}$-algebraic closure of
$a_i\mathbb{M}$,
  it  follows that
$\tp(\cl(a_{i}\mathbb{M})/\mathbb{N})$ is
a coheir of $\tp(\cl(a_{i}\mathbb{M})/\mathbb{M})$.
Thus by Lemma  \ref{lem:4.4} and Lemma \ref {lem:4.5} we have
$\cl(a_{i}\mathbb{N})=\cl(a_{i}\mathbb{M})\oplus_{\mathbb{M}} \mathbb{N}$.
This means that
$q_{1}(x)=q_{2}(x)$ if and only if there is an
$\mathcal{L}$-isomorphism between
$\cl(a_{1}\mathbb{M})$ and  $\cl(a_{2}\mathbb{M})$
mapping
$a_{1}$
to
$a_{2}$ and fixing $\mathbb{N}$ pointwise.
In other words,
$q_{1}(x)=q_{2}(x)$ if and only if  $\tp_{\mathcal{L}}(\cl(a_{1}\mathbb{M})/\mathbb{N})=\tp_{\mathcal{L}}(\cl(a_{2}\mathbb{M})/\mathbb{N})$.
\par
Now since $T$ is
dependent,
there are at most
$2^{|\mathbb{M}|}$ $\mathcal{L}$-coheirs of
$\tp_{\mathcal{L}}(\cl(a_{1}\mathbb{M})/\mathbb{M})$ over
$\mathbb{N}$.
On the other hand by the
above
argument
any  coheir extension of
$p(x)$
is uniquely determined by
$\tp_{\mathcal{L}}(\cl(a\mathbb{M})/\mathbb{N})$
for any realization
$a$ of
$p(x)$. Therefore there are at most
$2^{|\mathbb{M}|}$ coheirs of
$p(x)$.
\end{proof}
By  the proof of the theorem above,
the number of coheirs is the same in $T$
and $\mathbb{T}_\alpha$.
This gives also the following corollary on the stability.
\begin{cor}
If
$T$
is stable, then so is
$\mathbb{T}_{\alpha}$.
\end{cor}
\begin{proof}
If $T$ is stable, then
for any
$\mathbb{M}\leqslant \mathbb{N}\models \mathbb{T}_{\alpha}$ every $1$-type over
$\mathbb{M}$ has a unique coheir extension to $\mathbb{N}$, and
hence $\mathbb{T}_\alpha$ is also stable.
\end{proof}
Our next result is that
when $\alpha$ is rational, then the strong dependence also
transfers from $T$ to $\mathbb{T}_\alpha$.
Before stating the result,
let us recall
some basic definitions concerning
the strong dependence.
\par
Let
$T^{'}$ be an
  $\mathcal{L}^{'}$-theory
and
$X$  an index set.
Let
$I_{t}=\{\bar{a}_{ti}:\hspace*{2mm}i<\omega\}$,
for
$t\in X$,  be  sequences of tuples in a monster model of
$T^{'}$  such that
$|\bar{a}_{ti}|=|\bar{a}_{sj}|$, for each
$i, j\in \omega$ and $ t, s\in X$.
For  a small set of parameters
$A$, we say that sequences
$\{I_{t}:\hspace*{2mm} t \in X\}$
are \textit{ mutually
indiscernible} over
$A$ if each
sequence
$I_{t}$ is indiscernible over
$ A \cup \bigcup \{I_{s}:\hspace*{2mm}s\neq t\}  $.
\par
Let
$\Sigma(\bar{x})$ be a partial type over
$A$ and
$\kappa$ be a
(finite or infinite) cardinal. We say that
dp-rk$(\Sigma, A) <\kappa$ if for each family
$\{I_{t}:\hspace*{2mm} t<\kappa \}$ of mutually indiscernible sequences over
$A$ and for any realization
$\bar{b}$
of
$\Sigma$ , there is
$t < \kappa$ such that
$I_{t}$ is indiscernible over
$A\bar{b}$.
\par
A dependent theory $T^{'}$ is \textit{strongly dependent} if for
$\bar{x}$ of any length,
dp-rk$(\bar{x} =\bar{x}, \emptyset) <\aleph_{0}$  (for an account on strong-dependence see \cite{P15}).
\par
The following fact states some equivalent forms of strong dependence which are use for \Cref{th:4.13}.
\begin{fact}[ \cite{D16, Sh}]\label{fact:4.0.22}
The following  are equivalent for a complete theory $T^{'}$.
\begin{enumerate}
\item
$T^{'}$ is strongly dependent.
\item
For every indiscernible sequence $I=\{\bar{b}_{i}:\hspace*{3mm}i<\kappa\}$ with tuples $\bar{b}_{i}$ at most countable,
and every finite set $C$, there is a convex equivalence relation $\sim$ on $\kappa$ with finitely
many classes and such that $\tp(\bar{b}_{i}/C)$  depends only on the $\sim$-class of $i$.
In other words, the set
$\{\tp(\bar{b}_i/C):\hspace*{3mm}i\in \kappa\}$ is finite and
for each $i\in  \kappa$ the set
$\kappa (i)=\{j:\hspace*{3mm}\tp(\bar{b}_j/C)=\tp(\bar{b}_i/C)\}$
is a convex subset of $\kappa$.
\item
For every indiscernible sequence $\{\bar{b}_{i}:\hspace*{3mm}i\in I\}$ with tuples $\bar{b}_{i}$ at most countable,
and every finite set $C$, there is a convex equivalence relation $\sim$ on $I$ with finitely
many classes and such that $\{\bar{b}_{i}:\hspace*{3mm} i\in j/\sim\}$ is  indiscernible over $C$. In other words, in the above item
the sequence $I_i=\{\bar{b}_j:\hspace*{3mm} j\in \kappa (i)\}$ is indiscernible over $C$.
\end{enumerate}
\end{fact}
%%%%%%%%%%%%%%%%%%%%%%%%%%%%%%%%%%%%%%
\begin{lem}\label{lem:4.0.23}
Let $I=\{\bar{a}_{i}:\hspace*{3mm}i< \kappa\}$ be an $\mathcal{L}_p$-indiscernible sequence over $A$. Then
\begin{enumerate}
\item
$I$ is $\mathcal{L}_{p}$-indiscernible also  over $\cl(A)$.
\item
There is a suitable enumeration
$\bar{b}_{i}=\bar{a}_{i}\bar{c}_{i}$
 of
each $\langle \cl(A\bar{a}_{i})\rangle$ such that
$I^{'}=\{\bar{b}_{i}:\hspace*{3mm}i<\kappa\}$ is  an $\mathcal{L}_p$-indiscernible sequence
over $A$.
\end{enumerate}
\end{lem}
\begin{proof}
1. Let
$\{\bar{c}_i:\hspace*{3mm} i<\kappa\}$ be an $\mathcal{L}_p$-indiscernible sequence over $\cl(A)$ that realizes the $EM$-type of $I$.
This sequence has the same type over $A$ as the type of $I$,  and therefore, there is an $\mathcal{L}_p$-automorphism $\sigma:\mathfrak{C}\to\mathfrak{C}$ fixing $A$ pointwise such that $\sigma(\bar{c}_i)=\bar{a}_i$, for each $i<\kappa$. So,  $\{\bar{a}_i:\hspace*{3mm} i<\kappa\}$ is $\mathcal{L}_p$-indiscernible over $\sigma(\cl(A))=\cl(A)$.
\par
2. As $\langle\cl(A\bar{a}_i)\rangle \subseteq \dcl(\cl(A\bar{a}_i))$, it is enough to prove 2 for
the sequence $I'= \{\cl(A\bar{a}_i): \hspace*{3mm}i<\kappa\}$. Since
$\bar{a}_i$ and $\bar{a}_0$ have the same type
over $A$,
enumerate each
$\cl(A\bar{a}_i)$  in accordance with $ \cl(A\bar{a}_{0})$ using
the
$\mathcal{L}_p$-automorphism $\sigma_i$
that fixes $A$ pointwise and sends $\bar{a}_{0}$ to $\bar{a}_{i}$.
Denote this enumeration by $\bar{d}_{i}$ and put $J=\{\bar{d}_{i}:\hspace*{3mm} i< \kappa\}$.
By standard application of Ramsey's theorem,
let $I'=\{\bar{b}_{i}:\hspace*{3mm} i<\kappa\}$
be the $A$-indiscernible sequence
satisfying the
$EM$-type of $J$ over $A$.
Since each $\bar{b}_i$ contains an automorphic image of
$\bar{a}_i$, by strong homogeneity of $\mathfrak{C}$, we may assume that each $\bar{b}_i$ includes $\bar{a}_i$.
%
%Then the $EM$-type of the sequence $J=\{\bar{d}_{i}:\hspace*{3mm} i< \kappa\}$ over $A$ extends the $EM$-type of $I=\{\bar{a}_{i}:\hspace*{3mm} i<\kappa\}$. Hence by Standard Lemma  we may find  an $\mathcal{L}_p$-indiscernible sequence $I^{'}=\{\bar{b}_i:\hspace*{3mm}  i<\kappa\}$ which satisfies the $EM$-type of $J$ and includes an automorphic image of $I$ over $A$. So, by strong homogeneity of $\mathfrak{C}$, we may assume that each $\bar{b}_i$ includes $\bar{a}_i$.
 \par
We now need  to show that $\bar{b}_i$ is indeed an enumeration of $\cl(A\bar{a}_i)$.
Recall that $\bar{d}_i$ is already an enumeration of
$\cl(A\bar{a}_i)$.
Let $A_0$ be a finite subset of $A$ and
$(s_1,\dots, s_l)$ be a finite subtuple of $\bar{a}_0$. Then
$\cl^n(A_0s_1,\dots,s_l)=\{e_1,\dots,e_k\}$
is a finite subtuple of $\bar{d}_0$.
 Thus, $|\cl^n(A_0, \sigma_i(s_1),\dots,\sigma_i(s_l))|=k$, for each $i<\kappa$.
Now the following two facts are
covered by the
 $EM$-type of $J$. First that for each $i<\kappa$, $\bar{b}_i$ includes $\cl(A\bar{a}_i)$. Second,
 that for any finite subset $A_0$ of $A$ and finite subtuple $(t_1,\dots, t_l)$ of $\bar{a}_i$, there is a finite subtuple  $(f_1,\dots,f_k)$ of $\bar{b}_i$ such that $\cl^n(A_0 t_1,\dots, t_l)=\{f_1,\dots,f_k\}$. Hence $\bar{b}_i$ is an enumeration  of $\cl(A\bar{a}_i)$, as required.
\end{proof}
\begin{theo}($\alpha$ rational)\label{th:4.13}
If
$T$
is
strongly dependent, then so is
$\mathbb{T}_{\alpha}$.
\end{theo}
\begin{proof}
Let $I=\{\bar{b}_i:\hspace*{3mm} i<\kappa\}$ be an indiscernible sequence and $C$ a finite set, as
in \Cref{fact:4.0.22}.
  By $D$-local character and since $\alpha$ is rational, there exists a finite set $H\subseteq I$ such that $C\forkindep^{D}_H I$, and in particular
$C\forkindep^{D}_H \bar{b}_i$ for each $i$.
 So by \Cref{lem:4.4} for every $i<\kappa$, $\langle\cl(\bar{b}_iCH)\rangle=\langle\cl(\bar{b}_iH)\rangle\oplus_{\langle\cl(H)\rangle} \langle\cl(CH)\rangle$.
 Hence by \Cref{types-closures}, $\tp(\bar{b}_i/\cl(CH))$ depends only on $\tp(\bar{b}_i/\cl(H))$ and $\tp_{\mathcal{L}}(\langle\cl(\bar{b}_iH)\rangle/\cl(HC))$.
  We will
 find a convex partition of $\kappa$
 in each part of which these two are
 fixed, and the statement follows from \Cref{fact:4.0.22}.
%  EXPLAIN THAT WE TREAT EACH OF THESE.
% WE FIRST FIND A PARTION IN WHICH THE FIRST ONE IS FIXED AND THEN...
 \par
   Since $H$ is a finite subset of $I$ and $I$ is $\mathcal{L}_{p}$-indiscernible,
  there is a finite partition $J_1,\dots, J_k$ of $\kappa$ into convex sets such that each $I_j=\{\bar{b}_i: \hspace*{3mm} i\in J_j\}$ is indiscernible over $H$.
  So in particular     $\tp(\bar{b}_i/H)$--determined by
  $\langle\cl(\bar{b}_iH)\rangle$--
  is fixed for each $i\in J_j$ and $1\leqslant j\leqslant k$.
We may assume that each $J_j$ is infinite, as otherwise, the set $I_j$ can be added to $H$.
   \par
Since  $I_j$ is indiscernible over $H$, by \Cref{lem:4.0.23} it  is also  indiscernible over $\cl(H)$. Moreover again by \Cref{lem:4.0.23} we may assume that $I^{'}_j=\{\langle\cl(\bar{b}_iH)\rangle:\hspace*{3mm} i\in J_j\}$ is indiscernible over $H$.\par
 Now since $T$ is strongly dependent, by \Cref{fact:4.0.22},  there exists a convex equivalence relation $\sim_j$ over $J_j$ with finitely many equivalence classes such that $\{\langle\cl(\bar{b}_iH)\rangle: \hspace*{3mm}i\in \alpha/ \sim_j\}$ is $\mathcal{L}$-indiscernible over $\cl(CH)$. So, in particular, $\tp_{\mathcal{L}}(\langle\cl(\bar{b}_iH)\rangle /\cl(CH))$ is fixed on any equivalence class of $\sim_j$. Hence $\sim=\bigcup_{j=1}^k \sim_j$ forms a convex equivalence relation on $\kappa$ with finitely many equivalence classes such that $\tp(\bar{b}_i/\cl(HC))$ (and hence $\tp(\bar{b}_i/C)$) depends only on the  $\sim$-class of $i$.
\end{proof}
We will show in the following that in the particular case that $T$ defines a linear order,
 $\mathbb{T}_{\alpha}$ is  not strongly dependent for any
 irrational $\alpha$.
\begin{theo}
Assume that $T$ defines a linear order and $\alpha$ is irrational. Then
$\mathbb{T}_{\alpha}$
is not strongly dependent.
\end{theo}
\begin{proof}
Let  $<^{\mathfrak{C}}$ be the interpretation of the order defined by $T$.
We construct
$\mathcal{L}_{p}$-mutually indiscernible sequences
$\{I_{i}:\hspace*{2mm}i<\omega\}$ and
find $b_{0}\in \mathfrak{C}$ such that
$I_{i}=\{a_{ij}:\ j<\omega\}$ is not $\mathcal{L}_p$-indiscernible  over
$b_{0}$ for each
$i<\omega$.
\par
Let $J=\{c_{ij}:\ (i,j)\in\omega\times \omega\}$ be an $\mathcal{L}$-indiscernible sequence over $\emptyset$, ordered by  $\omega\times \omega$ with horizontal lexicographic order. For convenience,  we may assume the sequence $J$ is increasing with respect to $<^{\mathfrak{C}}$.  We
now inductively define a sequence $(D_n,E_n,F_n)_{ n<\omega}$ of subsets of $\mathfrak{C}$ and a sequence $(s_{n+1},k_{n+1})_{n<\omega}$ of tuples of natural numbers
which,  in particular, have the following properties:
\begin{enumerate}

\item   $E_n\subseteq J_n=\{c_{nj}:j\in \omega\}$ and $|E_{n+1}|=s_{n+1}$.
\item   $|F_{n+1}|=k_{n+1}-s_{n+1}$ and $F_{n+1}\subseteq \acl(D_n E_{n+1})$. Furthermore each $s_{n+1}$-element subset of $E_{n+1}F_{n+1}$ is a basis over $D_n$.

\item $D_0\subseteq D_1\subseteq \dots \subseteq D_n$,  $D_n\subseteq  \acl(\bigcup_{i=0}^{n} E_i)$ and  $D_n\cap J=\bigcup_{i=0}^n E_i$.

\end{enumerate}

\par
 For $n=0$, let $E_0=D_0=\{c_{01}\}$ and $F_0=\emptyset$. Suppose by induction that the sets $(D_n,E_n,F_n)$ are constructed and $D_n$ satisfies  condition $3$ above. By Dirichlet's  rational approximation theorem there is a pair of natural numbers $s_{n+1},k_{n+1}$ with $-(1-\alpha)/2^{n+1}< s_{n+1}- \alpha k_{n+1}< 0$. Let $E_{n+1}\subseteq J_{n+1
}$ be the set
$ \{c_{(n+1)1}, c_{(n+1)2},\ldots, c_{(n+1)s_n}\} $. Since $J$ is an indiscernible sequence over $\emptyset$, by Fact \ref{f:2.3}, $E_{n+1}$ is $\dim$-independent from $D_n$ over $\emptyset$. Hence $\dim(E_{n+1}/D_n)=s_{n+1}$.
By indecomposablity of $T$  let $F_{n+1}$ be the set consisting of   $k_{n+1}-s_{n+1}$ elements in $\acl(E_{n+1}D_n)$, with the property that each $s_{n+1}$-element subsets of $E_{n+1}F_{n+1}$ forms a basis  over $D_n$. Finally set $D_{n+1}=D_n\cup E_{n+1}\cup F_{n+1}$. Note that
$\dim(D_{n+1}/D_n)=s_n$ and $|D_{n+1}\setminus D_n|=k_n$. Furthermore,  $D_{n+1}\subseteq \acl(\bigcup_{i=0}^{n+1}E_i)$.
\par
 Now set $D=\bigcup_{i=0}^{\infty} D_i$ and $I=J\setminus  D$. Notice that  $I=J\setminus\bigcup_{i=0}^{\infty} E_i$. So $I$ is also an $\mathcal{L}$-indiscernible sequence over $\emptyset$.
  So we can re-enumerate $I$ and represent it by $\{d_{ij}:\hspace*{2mm}(i,j)\in\omega\times \omega\}$.
By the construction,
$D\subseteq \acl(\bigcup_{i=0}^{\infty} E_i)$.  On the other hand since $J$ is  $\mathcal{L}$-indiscernible, it follows from Fact \ref{f:2.3} that  $D\forkindep^{dim}_{\emptyset} I$ and $\acl(I)\cap\acl( D)=\acl(\emptyset)$.
\par
Let $H=\langle D\cup I \rangle$. Then $H= \langle D\oplus_{\emptyset}I\rangle =\langle D\rangle \oplus_{\emptyset} \langle I \rangle$. Now   we recolor $H$ by letting
only the elements in $ D$ colored and the rest non-colored. Hence the following conditions are fulfilled.
\begin{itemize}
\item For each $n$, $D_n\in \mathcal{K}_{\alpha}^{+}$, and hence $D\in \mathcal{K}_{\alpha}^{+}$.
\item $I,H\in \mathcal{K}_{\alpha}^{+}$.
\item  Every finite subset of $I$ as well as $I$ itself are closed  in $H$.
\item   $(D_n,D_{n+1})$ is a minimal pair, for each $n\in\omega$. Hence $\cl_H(c_{01})=D$.
\end{itemize}
Suppose that $f:H\to \mathfrak{C}$ is a strong embedding and $b_n=f(c_{n1})$ and $a_{ij}=f(d_{ij})$ for each $i,j,n\in\omega$. Then it is clear that $I_i=\{a_{ij}:\hspace*{2mm}j\in\omega\}$ are $\mathcal{L}_p$-mutually  indiscernible sequences over $\emptyset$. Furthermore
$a_{n0}<^{\mathfrak{C}}b_{n}<^{\mathfrak{C}}a_{n1}$ for each $n\in\omega$. We claim that non of the
$I_i$ is  $\mathcal{L}_{p}$-indiscernible over
$b_0$ for each
$i\in\omega$.
\par
Assume for a contradiction,  that $I_i$ is indiscernible over $b_0$, for some $i$ . Note that since $b_i\in \cl(b_0)\subseteq \Acl(b_{0})$, there exists an algebraic $\mathcal{L}_p$-formula $\psi(x,b_0)$ which is satisfied by $b_i$.
Let $\varphi(x,y,b_0)$ be the formula
\[\exists t \quad (x<^{\mathfrak{C}}t<^{\mathfrak{C}}y \wedge \psi
(t,b_0)).\] Then
$\varphi(a_{i0},a_{i1},b_0)$ holds. But since  $I_i$ is  indiscernible over $b_0$,
$\varphi(a_{ij},a_{ij+1},b_0)$ holds  for each $j\in\omega$.

But this gives infinitely many elements satisfying $\psi(x,b_0)$ which   contradicts  the algebraicity  of $\psi(x,b_0)$.

\end{proof}
\begin{cor}
If $\alpha$ is irrational then $\mathbb{ODAG}_{\alpha}$ and $\mathbb{RCF}_{\alpha}$ are not strongly dependent.
\end{cor}
Finally bring the paper to an end with a curious counterexample on the distality of $\mathbb{T}_{\alpha}$ assuming $T$ itself is distal.
\par
Recall from \cite{P13} that a dependent theory  $T^{'}$ is distal if for any indiscernible  sequence $I$,  every set $A$,
tuple $\bar{b}$ and $A$-indiscernible sequence $I^{'} = I_{1} + I_{2}$ with $I_{1}$ and $I_{2}$ without endpoints and  EM-$\tp(I^{'})$=EM-$\tp(I)$, if $I_{1} +\bar{b}+ I_{2}$ is indiscernible,
then it is $A$-indiscernible.
\par
Unlike the strong dependence  we will see that distality does not transfer to $\mathbb{T}_{\alpha}$,
both for rational and irrational $\alpha$,
 even when we start with
as distal a theory as the o-minimal theories.
\begin{exm}
Let $T$ be  any o-minimal expansion of ODAG. Then $\mathbb{T}_{\alpha}$ is not distal for any $0<\alpha\leqslant 1$.\\
Let  $A=\{a\}$ be a one-element set,  $I^{'}=I_{1}+I_{2}$  with both $I_{1}$ and $I_{2}$ sequences of elements in $\mathfrak{C}$  ordered by rational numbers and an element $b\in \mathfrak{C}$  such that  $I_{1}+b+I_{2}$ is $\mathcal{L}$-indiscernible  over $A$. Put $I=I_{1}$.  It follows  from the Fact \ref{morley-sequence} that  $I_{1}+b+I_{2}$ is a $\dim$-Morley  sequence over $A$. In particular $a+b\notin \langle I^{'} A\rangle$. Now to find an $\mathcal{L}_{p}$-indiscernible sequence  which violates the distality we color the  elements of $\mathcal{P}=\langle I^{'} \{b\}  A\rangle$ by letting $a+b$ colored and the rest  non-colored. Then it follows that $\mathcal{P}\in \mathcal{K}_{\alpha}^{+}$ and if $f:\mathcal{P}\to \mathfrak{C}$ is a strong embedding  then we have that  $f(I^{'})$ is an $\mathcal{L}_{p}$-indiscernible over $f(A)$ and $f(I_{1})+f(b) +f(I_{2})$ is an $\mathcal{L}_{p}$-indiscernible but it is not the case that $f(I_{1})+f(b) +f(I_{2})$ is an $\mathcal{L}_{p}$-indiscernible over
$f(A)$. Hence  this shows that $\mathbb{T}_{\alpha}$ is not distal.
\end{exm}

\end{document}